\documentclass[a4paper,10pt]{amsart}
\usepackage{amsthm,amsmath,amssymb,amsfonts}
\usepackage[T1]{fontenc}
\usepackage[utf8]{inputenc}
\usepackage{verbatim}
\usepackage{marvosym}
\usepackage{stackrel}
\usepackage{graphicx}
\usepackage{eucal}
\usepackage{bbold}
\usepackage[svgnames]{xcolor}
\definecolor{acid}{HTML}{438806}
\definecolor{crimson}{RGB}{177,12,12}
\definecolor{teal}{HTML}{007B7F}
\definecolor{orange}{HTML}{F28C28}
\definecolor{coral}{HTML}{FF6F61}
\definecolor{blue(munsell)}{rgb}{0.0, 0.5, 0.69}
\usepackage{lipsum}
\usepackage{svg}
\usepackage{bbm}
\usepackage{caption}
\usepackage{enumitem}
\usepackage{mathtools}
\usepackage{thmtools}
\usepackage{epigraph}
\usepackage{color}
\usepackage{microtype}
\usepackage{relsize}
\usepackage{amsfonts}
\usepackage{adjustbox}
\usepackage{framed}
\usepackage[bbgreekl]{mathbbol}
\usepackage{mathrsfs}
\usepackage{hyperref}
\hypersetup{
    colorlinks = true,
    linkbordercolor = {teal},
   	linkcolor ={teal},
	anchorcolor = {pink},
	citecolor =  {coral},
	filecolor = {blue},
	menucolor = {blue},
	runcolor =  {blue},
	urlcolor = {orange},
}
\usepackage{trimclip}

\DeclareFontFamily{U}{min}{}
\DeclareFontShape{U}{min}{m}{n}{<-> udmj30}{}
\newcommand{\yo}{\!\text{\usefont{U}{min}{m}{n}\symbol{'210}}\!}

\makeatletter
\def\@tocline#1#2#3#4#5#6#7{\relax
  \ifnum #1>\c@tocdepth % then omit
  \else
    \par \addpenalty\@secpenalty\addvspace{#2}%
    \begingroup \hyphenpenalty\@M
    \@ifempty{#4}{%
      \@tempdima\csname r@tocindent\number#1\endcsname\relax
    }{%
      \@tempdima#4\relax
    }%
    \parindent\z@ \leftskip#3\relax \advance\leftskip\@tempdima\relax
    \rightskip\@pnumwidth plus4em \parfillskip-\@pnumwidth
    #5\leavevmode\hskip-\@tempdima
      \ifcase #1
       \or\or \hskip 1em \or \hskip 2em \else \hskip 3em \fi%
      #6\nobreak\relax
    \dotfill\hbox to\@pnumwidth{\@tocpagenum{#7}}\par
    \nobreak
    \endgroup
  \fi}
\makeatother

\usepackage{stmaryrd}
\usepackage{tikz}
\usepackage{tikz-cd}
\usepackage{quiver}

\usepackage[capitalise]{cleveref}

\theoremstyle{definition}
\newtheorem{thm}{Theorem}[subsection]
\newtheorem*{thm*}{Theorem}
\newtheorem{prop}[thm]{Proposition}
\newtheorem*{prop*}{Proposition}
\newtheorem{lem}[thm]{Lemma}
\newtheorem{cor}[thm]{Corollary}
\newtheorem{defn}[thm]{Definition}
\newtheorem*{defn*}{Definition}

\newtheorem*{war*}{Warning}
\newtheorem{rem}[thm]{Remark}
\newtheorem{constr}[thm]{Construction}
\newtheorem{exa}[thm]{Example}

\newtheorem*{notat*}{Notation}

\newcommand{\textdel}[1]{}

\newcommand{\Ecal}{\mathcal{E}}
\newcommand{\Fcal}{\mathcal{F}}

\newcommand{\Hcal}{\mathcal{H}}
\newcommand{\Xcal}{\mathcal{X}}
\newcommand{\Ocal}{\mathcal{O}}

\newcommand{\Set}{\mathsf{Set}}
\newcommand{\op}{{}^{\mathrm{op}}}
\DeclareMathOperator{\id}{id}

\newcommand{\nocontentsline}[3]{}
\newcommand\stoptoc{%
   \let\origcontentsline\addcontentsline
   \let\addcontentsline\nocontentsline
}
\newcommand\resumetoc{%
   \let\addcontentsline\origcontentsline
}

\newcommand{\mc}[1]{\mathcal{#1}}

\newcommand{\mbb}[1]{\mathbb{#1}}
\newcommand{\mr}[1]{\mathrm{#1}}

\newcommand{\ms}[1]{\mathsf{#1}}
\newcommand{\lan}{\ms{lan}}
\newcommand{\ran}{\ms{ran}}

\newcommand{\WRInj}{\ms{WRInj}}
\newcommand{\Topoi}{\ms{Topoi}}
\newcommand{\LEX}{\ms{LEX}}
\newcommand{\Lex}{\ms{Lex}}

\newcommand{\CAT}{\ms{CAT}}

\newcommand{\psh}{\ms{Psh}}
\newcommand{\sh}{\ms{Sh}}
\newcommand{\alg}{\ms{Alg}}
\newcommand{\surj}{\twoheadrightarrow}
\newcommand{\inj}{\rightarrowtail}
\newcommand{\hook}{\hookrightarrow}

\newcommand{\stt}[1]{\{\,#1\,\}}

\newcommand{\dv}{\uparrow}

\newcommand{\set}[1]{\{#1\}}
\newcommand{\Cl}{\ms{Cl}}

\makeatother
\makeatletter
\newcommand{\lt@}[2]{%
  \vtop{\m@th\ialign{##\cr
    \hfil$#1\operator@font lim$\hfil\cr
    \noalign{\nointerlineskip\kern1.5\ex@}#2\cr
    \noalign{\nointerlineskip\kern-\ex@}\cr}}%
}
\newcommand{\lt}{%
  \mathop{\mathpalette\lt@{\leftarrowfill@\textstyle}}\nmlimits@
}
\makeatother

\newcommand{\Syn}{\mathsf{Syn}}
\newcommand{\T}{\mathsf{T}}
\newcommand{\Sh}{\mathsf{Sh}}
\newcommand{\Obb}{\mathbb{O}}
\newcommand{\Psh}{\ms{Psh}}
\newcommand{\Alg}{\ms{Alg}}
\newcommand{\lex}{\ms{Lex}}
\newcommand{\Sub}{\operatorname{Sub}}

\DeclareMathOperator*{\colim}{colim}
\newcommand{\Hflat}{\Hcal_{\text{flat}}}
\newcommand{\Hmatte}{\Hcal_{\text{matte}}}
\newcommand{\Hdom}{\Hcal_{\text{dom}}}
\newcommand{\Hpure}{\Hcal_{\text{pure}}}

\newcommand*\cocolon{%
	\nobreak
	\mskip4mu plus1mu
	\mathpunct{}%
	\nonscript
	\mkern-\thinmuskip
	{:}%
	\mskip2mu
	\relax
}

\newcommand{\ult}{%
  \mathrel{\tikz [line width=.11ex, double distance=.33ex]
    \draw[>-]
    (0.3,0) -- (0,0);}
}

\def\slashedrightarrow{\relbar\mathrel{\mkern-4mu}\joinrel\mapstochar\mathrel{\mkern-4mu}\joinrel\rightarrow}
\newcommand{\pro}{\slashedrightarrow}

\makeatletter
  \def\title@font{\MakeUppercase}
  \let\ltx@maketitle\@maketitle
  \def\@maketitle{\bgroup%
    \let\ltx@title\@title%
    \def\@title{\resizebox{\textwidth}{!}{%
      \mbox{\title@font\ltx@title}%
    }}%
    \ltx@maketitle%
  \egroup}
\makeatother

\title[Conceptual completeness for subgeometric logics]{Conceptual completeness for subgeometric logics}
\author{Ivan Di Liberti}
\author{Umberto Tarantino}
\author{Lingyuan Ye}

\address{
Ivan \textsc{Di Liberti} \newline
Department of Philosophy, Linguistics and Theory of Science\newline
University of Gothenburg\newline
Gothenburg, Sweden\newline
\href{mailto:diliberti.math@gmail.com}{\sf diliberti.math@gmail.com}
}

\address{
Umberto \textsc{Tarantino} \newline
Université Paris Cité, CNRS, IRIF, F-75013\newline
Paris, France\newline
\href{mailto:tarantino@irif.fr}{\sf tarantino@irif.fr}
}

\address{
Lingyuan \textsc{Ye} \newline
Department of Computer Science\newline
University of Cambridge\newline
Cambridge, UK\newline
\href{mailto:ye.lingyuan.ac@gmail.com}{\sf ye.lingyuan.ac@gmail.com}
}

\thanks{The first-named author was supported by the Swedish Research Council (SRC, Vetenskapsrådet) under Grant No.~2019-04545. The research has received funding from Knut and Alice Wallenbergs Foundation through the Foundation's program for mathematics. The second-named author acknowledges financial support from the Agence Nationale de la Recherche
(ANR), project ANR-23-CE48-0012-01.}

\begin{document}

\maketitle

\vspace{-2em}
\begin{abstract}
\looseness=-1
We explore the notion of conceptual completeness for a fragment of geometric logic in the framework developed by the first and third author. Unlike its traditional interpretation as a reconstruction of syntax from semantics, in this paper we characterise conceptual completeness of a fixed fragment in terms of a duality between theories and topoi. We then show that conceptually complete fragments are conservatively embedded in full geometric logic, thus casting conceptual completeness in a new proof-theoretic light. We give a new proof of conceptual completeness for coherent logic, and we also show that regular, disjunctive, and essentially algebraic logic with falsum are conceptually complete. Finally, we show that our notion is equivalent to a traditional reconstruction result under the assumption of completeness with respect to set-based models: in the coherent case, we thus recover Makkai's original reconstruction theorem via ultracategories. 

\end{abstract}

   {
   \hypersetup{linkcolor=black}
    \tableofcontents
}

   \newpage
\section{Introduction}

\stoptoc 

\subsection*{Overview}
A traditional completeness theorem, for a logic, expresses that the class of models of any theory $\mathbb T$ in the logic is able to \emph{separate} formulas not provably equivalent modulo $\mathbb T$. In 1987, Makkai proved his celebrated \emph{(strong) conceptual completeness}\footnote
{
  In this paper, we will always refer to the result originally proved in \cite{makkaiStoneDualityFirst1987} and  dubbed \emph{strong conceptual completeness} in \cite{makkaiStrongConceptualCompleteness1988}, instead of the weaker result originally proved in \cite{makkaiFirstOrderCategorical1977} under the name of \emph{conceptual completeness}, see below. 
}
theorem for coherent first-order logic \cite{makkaiStoneDualityFirst1987}: strengthening the usual completeness, his result shows that we can \emph{reconstruct} the syntax of $\mathbb T$ from its semantics, conceived as the category $\mathsf{Mod}(\mathbb T)$ of models and their homomorphisms, provided that the latter is endowed with some additional structure. In this sense, conceptual completeness is traditionally intended as a \emph{duality} between syntax and set-based semantics. The essence of the theorem is encoded by a striking \emph{definability} property: a functor $\mathsf{Mod}(\mathbb T) \to \Set$ is defined by a formula in the theory\footnote
{
    Up to \emph{elimination of imaginaries} in the sense of Shelah \cite{shelahClassificationTheoryNumber1990}, see \cite{harnikModelTheoryVs2011}.
}, in the sense that it maps each model $M$ to the interpretation $\llbracket \phi \rrbracket_M$ of some fixed formula $\phi(x_1,\dots, x_n)$ in $M$, if and only if it preserves \emph{ultraproducts}.

At the propositional level, where categories of models are simply \emph{sets}, such a reconstruction result is usually better known as a \emph{representation theorem}: e.g., for Boolean algebras \cite{stoneTheoryRepresentationBoolean1936} in the classical case, or for distributive lattices in the coherent case \cite{stoneTopologicalRepresentationsDistributive1938,priestleyRepresentationDistributiveLattices1970}. In both cases, the key additional structure on the set of models is given by an appropriate \emph{topology}. Identifying a theory $\mathbb T$ with its Lindenbaum-Tarski algebra $\mathrm{A}_{\mathbb T}$, models correspond to morphisms $\mathrm{A}_{\mathbb T}\to \mathsf{2}$; each theorem can then be phrased as the statement that a dual adjunction, arising by the two-element set $\mathsf{2}$ as a dualising object, is a \emph{reflection}. 

\looseness=-1
At the level of predicate logic, a toy instance of such a reconstruction theorem is given, in the context of \emph{essentially algebraic logic}, by \emph{Gabriel-Ulmer duality} \cite{gabrielLokalPrasentierbareKategorien1971}. In this case, the syntax is simple enough that it can be recovered from the semantics without the need for additional structure. The role of $\mathsf{2}$ in detecting models is played here by the category $\Set$ of sets and functions: indeed, identifying a theory $\mathbb T$ with a finitely-complete category $\mc C_{\mathbb T}$, models are finite-limit-preserving functors $\mc C_{\mathbb T} \to \Set$. Then, $\mc C_{\mathbb T}$ can be recovered as the full subcategory of $\mathsf{Mod}(\mathbb T)$ spanned by those functors which preserve \emph{all} (small) limits and \emph{filtered} colimits. In formal analogy with the propositional case, the reconstruction can also be expressed by stating that an appropriate dual adjunction, induced by $\Set$ as a dualising object, is a reflection.

\looseness=-1
In full first-order logic, some additional structure on the semantics is really necessary to recover the syntax. For instance, it is well-known that two $\aleph_0$-categorical theories may fail to be bi-interpretable even if the automorphism groups of their countable model are isomorphic, see e.g.\ \cite{evansCounterexamplesConjectureRelative1990}.\footnote{In fact, for the two theories to be bi-interpretable, the two automorphism groups need to be isomorphic \emph{as topological groups} once endowed with an appropriate topology, see \cite[\S 1]{ahlbrandtQuasiFinitelyAxiomatizable1986}.} Makkai's crucial insight is that the category of models of a first-order theory $\mathbb T$ naturally carries additional structure, determined by ultraproducts and hence dubbed \emph{ultrastructure}. In analogy with the propositional case --- based on the fact that topologies can be described in terms of ultrafilter convergence, cf.\ \cite{manesTripleTheoreticConstruction1969, barrRelationalAlgebras1970} --- we can see ultrastructure as the extra `topological' structure on $\mathsf{Mod}(\mathbb T)$ needed to recover the syntax. As above, the reconstruction result for coherent first-order logic can also be formulated in terms of a reflective dual adjunction\footnote{
    To be more precise, the adjunction need only be a \emph{reflection in the small}, see \cite[\S 8]{makkaiStoneDualityFirst1987}. 
} induced by $\Set$ as a dualising object, between \emph{pretoposes} and categories equipped with ultrastructure, that is, \emph{ultracategories}. Makkai's theorem was later extended by Lurie \cite{lurieUltracategories2018}, who identifies coherent first-order theories with their \emph{classifying topoi}; other approaches to a syntax-semantics duality for first-order theories include \cite{awodeyFirstorderLogicalDuality2013} via topological groupoids, \cite{breinerSchemeRepresentationFirstorder2014} via schemes, and \cite{yeDualityTheoryCategorical2026} via genuine topologies in restricted cases. Also identifying syntax with topoi, the independent works \cite{saadiaExtendingConceptualCompleteness2025,hamadGeneralisedUltracategoriesConceptual2025,vangoolToposesEnoughPoints2026} generalised ultracategories in order to encompass categories of models of \emph{geometric} (predicate) logic and obtain a reconstruction theorem for geometric theories with enough set-based models.

\looseness=-1
Besides its importance at an abstract level, such a reconstruction result also has concrete model-theoretic applications. For instance, already in the propositional case, usually a version of \emph{Craig interpolation} can be deduced by looking at dual properties on spaces of models. Makkai's result, in particular, entails that a functorial choice of a subset in each model of a coherent first-order theory in one sort is definable if and only if it is closed under ultraproducts, and more definability results are derived in \cite{makkaiDualityDefinabilityFirst1994}. Moreover, it entails that, roughly speaking, an interpretation $I \colon \mathbb T \to \mathbb T'$ which induces an equivalence of categories $\mathsf{Mod}(\mathbb T') \to \mathsf{Mod}(\mathbb T)$ is itself a bi-interpretation: this (weaker) property, originally proved in \cite{makkaiFirstOrderCategorical1977} and studied by Pitts in \cite{pittsInterpolationConceptualCompleteness1986, pittsConceptualCompletenessFirstorder1989}, is also known as ``conceptual completeness''.  Building on \cite{makkaiStoneDualityFirst1987}, Zawadowski \cite{zawadowskiDescentDuality1995} proved a descent theorem for pretopoi which implies that, given a \emph{conservative} extension $\mathbb T'$ of a theory $\mathbb T$, the latter can be recovered from those models which can be expanded to models of $\mathbb T'$; his results also imply an interpolation result for coherent first-order logic.

\subsection*{In this paper} The crucial observation motivating the present work comes from \cite{dilibertiGeometryCoherentTopoi2022}: for a geometric\footnote{From here on, we work exclusively with predicate logic, cf.\ \cite[Rem.\ 5.0.5]{dilibertiLogicConcepts2category2025}.} theory $\mathbb T$, the property that its category of models is closed under ultraproducts can be rephrased as the existence of certain \emph{right Kan extensions} in the 2-category $\Topoi$ of topoi and geometric morphisms.  In other words, while an ultrastructure on a category is truly \emph{additional structure} imposed on it, it becomes a \emph{property} expressible in purely categorical terms once we restrict to categories of models of geometric theories, in the same way as a topology is additional structure on a set while being compact and Hausdorff is a property of a topological space. This insight leads to the idea of \emph{semantic prescriptions} developed in \cite{dilibertiLogicConcepts2category2025}, where the framework of \emph{right Kan injectivity} in $\Topoi$ is employed to achieve a fully modular treatment of subgeometric logics. 

In this setting, conceptual completeness for coherent logic boils down to a definability result akin to the one proved by Makkai: for a small pretopos $\mathcal A$ with associated coherent topos $\sh(\mc A, J_{\mr{coh}})$, a geometric morphism $\sh(\mc A, J_{\mr{coh}}) \to \Set[\Obb]$ is represented by an object of $\mc A$ if and only if it preserves an appropriate class of Kan extensions. This rephrasing reveals the aim and perspective of this paper: here, we advocate for a new interpretation of a conceptual completeness theorem as a \emph{syntactic} phenomenon, a duality result between syntax --- i.e., theories --- and syntax itself --- i.e., their classifying topoi. To this aim, instead of an intuitive ``reconstruction'' of syntax from semantics, we start from the definition of conceptual completeness given in \cite{dilibertiLogicConcepts2category2025} as a property that a \emph{fragment of geometric logic} may or may not hold, formalised in the language of Kan injectivity. Following the philosophy of \emph{ibid.}, this definition allows us for a {modular} study of conceptual completeness across subgeometric logics which appears to be completely original in the literature, where each logic is treated independently and with \emph{ad hoc} techniques. 

Here, the essence of conceptual completeness is detached from set-based semantics and a reconstruction theorem, in a two-fold sense. On one side, a conceptual completeness result for a fixed fragment is now liberated from the quest for the appropriate extra structure on categories of models so as to reconstruct syntax. On another side, such a result does not entail the completeness of the fragment, in the traditional sense, with respect to set-based models: the only engrained completeness is with respect to models in arbitrary topoi. Together with the leading role played by the \emph{object classifier} $\Set[\Obb]$, instead of $\Set$, this shows how our work is actually rooted in \emph{topos}-based semantics, identified with syntax by means of classifying topoi. In the paper, we will see how this notion of conceptual completeness recovers a familiar reconstruction theorem \emph{under the assumption} of completeness with respect to set-based models; moreover, to further our syntactic point, we will also realise that the theories in a conceptually complete fragment of geometric logic satisfy a suitable \emph{conservativity} property with respect to full geometric logic.

\subsection*{Contributions and structure of the paper}We begin by recalling the necessary background on the theory developed in \cite{dilibertiLogicConcepts2category2025} about fragments of geometric logic $\Hcal$ and the associated constructions defined in terms of right Kan injectivity:
\begin{itemize}[label={--}]
\item the 2-category $\WRInj(\Hcal)$ of topoi \emph{formally belonging} to the logic $\Hcal$ and their morphisms,
\item the \emph{syntactic category} $\Syn^{\Hcal}(\Xcal)$ of a topos $\Xcal$ formally in $\Hcal$,
\item the \emph{monad of syntax} $\T^{\Hcal}$, and
\item the \emph{classifying topos} $\Cl^{\Hcal}(\mc A)$ of a $\T^{\Hcal}$-algebra $\mc A$.
\end{itemize}
\looseness=-1
With these notions at hand, our investigations in this paper begin where \cite{dilibertiLogicConcepts2category2025} ended, on a first definition of what it means for a fragment of geometric logic to be conceptually complete.

\begin{defn*}[\protect{\cite[Def.\ 6.3.7]{dilibertiLogicConcepts2category2025}}]
    A bounded logic $\Hcal$ is \emph{conceptually complete} if the 2-functor $\Cl^{\Hcal} \colon \alg(\T^{\Hcal})\op \to \WRInj(\Hcal)$ is 2-fully-faithful.
\end{defn*}

\looseness=-1
In \Cref{sec:everyalgisasyncat}, we derive an equivalent description of conceptual completeness (\Cref{cor:cc-as-every-algebra-is-syntactic-category}). Concretely, this is achieved by realising that the classifying topos construction is right adjoint to the syntactic category construction (\Cref{prop:adjunction-syn-cl}), so that its fully-faithfulness can be rephrased in terms of the adjunction $\Syn^{\Hcal}\dashv \Cl^{\Hcal}$ being a \emph{reflection}, i.e.\ its counit being invertible. This yields a modular property in the language of Kan injectivity which, in the coherent case, captures the core of Makkai's conceptual completeness theorem as described above.

\begin{thm*}[\ref{cor:cc-as-every-algebra-is-syntactic-category}]
  A bounded logic $\Hcal$ is conceptually complete if and only if, for each small $\T^{\Hcal}$-algebra $\mc A$, the corestriction of the Yoneda embedding $\mc A \hook \Syn^{\Hcal}(\Cl^{\Hcal}(\mc A))$ is essentially surjective.
\end{thm*}

The dual adjunction $\Syn^{\Hcal}\dashv \Cl^{\Hcal}$ can be conceptualised in two ways. On one hand, from a \emph{logical} perspective, it can be seen as a kind of \emph{Diaconescu's equivalence}, describing morphisms of $\Hcal$-injectives into $\Cl^{\Hcal}(\mc A)$, for a small $\T^{\Hcal}$-algebra $\mc A$, in terms of morphisms of $\T^{\Hcal}$-algebras out of $\mc A$ landing into a syntactic category. In this sense, it captures the essence of conceptual completeness as a duality between syntax and topos-based semantics. On the other hand, from a more \emph{geometric} perspective, $\Syn^{\Hcal}\dashv \Cl^{\Hcal}$ can be seen as a duality between \emph{sites} and \emph{categories of sheaves} for topoi formally in $\Hcal$, where the syntactic category of such a topos intuitively plays the role of a site of presentation for it.

In \Cref{sec:examples} we showcase the power and flexibility of our theory by proving conceptual completeness of four logics of interest: \emph{coherent} logic (\Cref{ssec:cc-coherent-logic}), \emph{regular} logic (\Cref{ssec:cc-regular-logic}), \emph{essentially algebraic} logic \emph{with falsum} (\Cref{ssec:cc-ess-alg-logic-w-falsum}), and \emph{finitary disjunctive} logic (\Cref{ssec:cc-finit-disj-logic}). Although the proof techniques in this section are more `specific' for each fragment and not fully modular\footnote{Hopefully, \emph{not yet} fully modular.}, we highlight how the amount of \emph{ad hoc} computations is substantially smaller than in previous case-by-case approaches in the literature, both on a conceptual and on a practical level. Moreover, the framework of Kan injectivity often \emph{simplifies} these arguments, compared to previous proofs, precisely as it avoids complicated structure artificially imposed on the semantics. 

\looseness=-1
In \Cref{sec:cons-emb}, we pursue our general analysis of conceptual completeness by studying those logics $\Hcal$ for which the adjunction $\Syn^{\Hcal}\dashv \Cl^{\Hcal}$ is \emph{idempotent}. We name these logics \emph{conservatively embedded} (in geometric logic), and we introduce some explicit syntax in order to justify this name in terms of conservativity of suitable interpretations (\Cref{prop:ce-logics-via-syntax}). Clearly, every conceptually complete logic is conservatively embedded: this entails that a conceptual completeness result has direct repercussions on the  \textit{proof theory} of the fragment.

\looseness=-1
Finally, in \Cref{sec:makkai-all-along}, we discuss how our notion of conceptual completeness relates to Makkai's notion in the coherent case. Towards this goal, we place ourselves in the framework of \emph{virtual ultracategories}, introduced in \cite{saadiaExtendingConceptualCompleteness2025,hamadGeneralisedUltracategoriesConceptual2025,vangoolToposesEnoughPoints2026} and developed also in \cite{aristoteProfunctorialAlgebras2026}, as the appropriate \emph{arena of formal model theory} against which to formalise the blueprint of semantic prescriptions. More concretely, we first show how these semantic prescriptions can be captured as Kan injectivity conditions in the 2-category of virtual ultracategories and virtual ultrafunctors, giving rise to an equivalence with the corresponding injectivity class in $\Topoi$ (\Cref{prop:formal-equivalence-injectivity-vult}). In the coherent case, this yields a reformulation of Makkai's result in terms of a composite adjunction, depicted below, being a reflection:
\[\begin{tikzcd}
	{\mathsf{Pretopoi}\op} & {\WRInj(\Hcal_{\beta})_{\ms{wep}}} & {\WRInj(\mc M_\beta)^{\ms{bnd}}_{\ms{sob}}}
	\arrow[""{name=0, anchor=center, inner sep=0}, "{{\Cl^{\Hcal_{\beta}}}}"', curve={height=18pt}, from=1-1, to=1-2]
	\arrow[""{name=1, anchor=center, inner sep=0}, "{{{{\Syn^{\Hcal_{\beta}}}}}}"', curve={height=18pt}, from=1-2, to=1-1]
	\arrow[""{name=2, anchor=center, inner sep=0}, "{{{{\mathsf{pt}}}}}"', curve={height=18pt}, from=1-2, to=1-3]
	\arrow[""{name=3, anchor=center, inner sep=0}, "{{{{\mc O}}}}"', curve={height=18pt}, from=1-3, to=1-2]
	\arrow["\dashv"{anchor=center, rotate=-90}, draw=none, from=1, to=0]
	\arrow["\simeq"{description}, draw=none, from=3, to=2]
\end{tikzcd}\]
\looseness=-1
This diagram captures our point of view on conceptual completeness, detaching it from a traditional completeness result. Indeed, the first adjunction being a reflection is equivalent to coherent logic being conceptually complete in our sense, while the second adjunction is an equivalence by virtue of completeness of coherent logic with respect to set-based models. In particular, we deduce that conceptual completeness of coherent logic in our sense is equivalent to Makkai's result (\Cref{cor:recovering-makkai}).

As a concluding remark, we abstract the previous line of reasoning beyond the coherent case. We introduce a notion of \emph{conceptual completeness à la Makkai} formalising a reconstruction theorem for a generic logic defined in terms of semantic prescriptions (\Cref{def:cc-à-la-makkai}), and we show that it is equivalent to conceptual completeness in our sense under the assumption of completeness property with respect to set-based models (\Cref{bigthoeremwithtrivialproof}). This completes our analysis of conceptual completeness as a syntactic phenomenon, a priori detached from set-based semantics (cf.\ \Cref{rem:resolution}).

\begin{thm*}[\ref{bigthoeremwithtrivialproof}]
	Let $\Hcal$ be a bounded logic such that $\Hcal = \mc O(\mc M)$ for some class $\mc M$ of morphisms in $\mathsf{vUlt^{bnd}}$, and $\Hcal$-classifying topoi have enough points.
Then, the following are equivalent:
\begin{enumerate}
	\item $\Hcal$ is conceptually complete.
	\item $\Hcal$ is conceptually complete \emph{à la} Makkai.
\end{enumerate}
\end{thm*}

\resumetoc

\begin{notat*}
\looseness=-1
    Throughout the paper, we will use the following notations and conventions. When speaking of 2-dimensional category theory, we will loosely speak of \emph{2-category}, \emph{2-functor}, \emph{2-adjunction} even for \emph{bicategories}, \emph{pseudofunctors}, \emph{pseudoadjunctions}. When dealing with 2-categories whose objects are themselves categories, we will typically use lowercase letters for the 2-category defined by \emph{small} objects, and capital letters for the 2-category defined by \emph{locally-small} objects: for instance, $\Lex$ (resp.\ $\LEX$) will denote the 2-category of small (resp.\ locally-small) finitely-complete categories, finite-limit-preserving functors, and natural transformations. This does not apply to $\Topoi$, which will denote the 2-category of (Grothendieck) topoi, geometric morphisms, and geometric transformations, nor to sub-2-categories thereof.
\end{notat*}

\subsection{On logics and doctrines} \label{ssec:logics-and-doctrines}
\looseness=-1
In this subsection, we recall the necessary background from \cite{dilibertiLogicConcepts2category2025} on which the rest of the paper is based.

The crucial definition, originating the theory developed also in \cite{dilibertiCraigInterpolationSubgeometric2026} and in this paper, is to identify the intuitive notion of a \emph{fragment of geometric logic} by requiring the existence of a class of structures in the category of (generalised) points of a theory. It is the insight of~\cite{dilibertiGeometryCoherentTopoi2022} that the relevant structures can be prescribed via a \emph{class of geometric morphisms}, by looking at \emph{(weak) right Kan injectivity} (cf.~\cite{dilibertiKZPseudomonadsKanInjectivity2024}) with respect to them in the 2-category $\Topoi$ of topoi. 

\begin{defn}[Logic]\looseness=-1
    A \emph{logic} $\Hcal$ is a class of geometric morphisms between topoi. 
\end{defn}

A logic then determines a class of topoi and geometric morphisms between them via Kan injectivity. These are the topoi whose generalised points admit the corresponding structure prescribed by the logic, in a sense developed in~\cite{dilibertiLogicConcepts2category2025}.

\begin{defn}[Topoi formally in the logic]
We say that a topos $\Xcal$ \emph{formally belongs} to a logic $\Hcal$ if it is (weakly) right Kan injective with respect to $\Hcal$-morphisms, that is, if the right Kan extension below exists in the 2-category $\Topoi$, for all $f \colon \mc E \to \mc F$ in $\Hcal$ and all geometric morphisms $x \colon \mc E \to \mc X$:
\[\begin{tikzcd}
	{\mc E} & {\mc X} \\
	{\mc F}
	\arrow["x", from=1-1, to=1-2]
	\arrow["f"', from=1-1, to=2-1]
	\arrow["{\ran_fx}"', dashed, from=2-1, to=1-2]
\end{tikzcd}\]
A \emph{morphism of $\Hcal$-injectives} is then a geometric morphism $g \colon \mc X \to \mc Y$ between topoi formally in $\Hcal$ which preserves these Kan extensions. We denote by $\WRInj(\Hcal)$ the locally full 2-category spanned by topoi formally in $\Hcal$ and morphisms of $\Hcal$-injectives.
\end{defn}

\begin{exa}[In this paper]\label{exa:examples-of-logics}
We now introduce the four logics which we will prove to be conceptually complete in \Cref{sec:examples}.
\begin{itemize}
    \item[i.]\label{def:coherent-logic}\emph{Coherent logic} is the class $\Hflat$ of \emph{flat} geometric morphisms, i.e.\ those geometric morphisms whose direct image preserves finite coproducts and epimorphisms.
    \item[ii.]\label{def:regular-logic}\emph{Regular logic} is the class $\Hmatte$ of \emph{matte} geometric morphisms, i.e.\ those geometric morphisms whose direct image preserves epimorphisms. 
    \item[iii.]\label{def:ess-alg-logic-with-falsum} \emph{Essentially algebraic logic with falsum} is the class $\Hdom$ of \emph{dominant} geometric morphisms, i.e.\ those geometric morphisms whose direct image preserves the initial object.
    \item[iv.]\label{def:finit-disj-logic}
\emph{Finitary disjunctive logic} is the class $\Hpure$ of \emph{pure} geometric morphisms, i.e.\ those geometric morphisms whose direct image preserves finite coproducts.
\end{itemize}
\end{exa}

\begin{rem}[Free topoi and essentially algebraic logic]
\looseness=-1
As observed already in \cite{johnstoneInjectiveToposes1981}, we recall here that \emph{free topoi}, i.e.\ presheaf topoi $\Psh(\mc C)$ over categories $\mc C$ in $\lex$, are (weakly) right Kan injective with respect to any geometric morphism, and hence they formally belong to any logic. In logical terms, this means that the `smallest' fragment of geometric logic that we can consider, corresponding to the class $\Hcal_{\mr{all}}$ of \emph{all} geometric morphisms, is \emph{essentially algebraic logic} (cf.\ \cite[Cor.\ 1.3.6]{dilibertiLogicConcepts2category2025}).
\end{rem}

\looseness=-1
Next, every topos $\Xcal$ which formally belongs to a logic $\Hcal$ admits a \emph{syntactic category}, playing the role of a tentative axiomatisation of the theory corresponding to the topos.

\begin{defn}[Syntactic categories]
    For a logic $\Hcal$ and a topos $\Xcal$ formally in $\Hcal$, the \emph{$\Hcal$-syntactic category} of $\Xcal$ is the category
    \[ \Syn^{\Hcal}(\mc X) \coloneqq \WRInj(\Hcal)(\Xcal, \Set[\Obb]),\]
    which we can identify with a full subcategory of $\Xcal$ as $\Set[\Obb] \coloneqq \Psh(\ms{FinSet}\op)$ is the \emph{object classifier} (cf.\ \cite[\S D3.2]{johnstoneSketchesElephantTopos2002}).
    
    This construction defines a 2-functor $\Syn^{\Hcal} \colon \WRInj(\Hcal)\op \to \LEX$.
\end{defn}

In the following, we will essentially only consider \emph{bounded} logics, in the sense below. This size restriction is motivated by practical and expository purposes rather than a real need: we believe that most of our results could be lifted to \emph{Morita-bounded} theories in the sense of \cite[Def.\ 5.1.6]{dilibertiLogicConcepts2category2025}, and at the moment we have no examples of logics which are not Morita-bounded.

\begin{defn}[Bounded logics]
A logic $\Hcal$ is \emph{bounded} if the syntactic category $\Syn^{\Hcal}(\Xcal)$ is small for each topos $\Xcal$ formally in $\Hcal$. 
\end{defn}

To each logic $\Hcal$ we can canonically associate a \emph{doctrine} $\T^{\Hcal}$, meaning here a lax-idempotent relative pseudomonad over the inclusion $\Lex \hook \LEX$ which is moreover a sub-monad of the presheaf construction $\Psh\colon \Lex \to \LEX$. Intuitively, $\T^{\Hcal}$ performs a completion under property-like natural operations respecting the \emph{semantic prescriptions} specified by the logic $\Hcal$. 
 
\begin{defn}[Monad of syntax]
    For a bounded logic $\Hcal$, its \emph{monad of $\Hcal$-syntax} $\T^{\Hcal}$ is the doctrine defined by setting, for each category $\mc C$ in $\Lex$:
    \[ \T^{\Hcal}(\mc C) \coloneqq \Syn^{\Hcal}(\Psh(\mc C)) = \WRInj(\Hcal)(\Psh(\mc C), \Set[\Obb])\]
\end{defn}

Every $\T^{\Hcal}$-algebra can be `completed' to a topos, classifying the theory in the logic $\Hcal$ represented by the algebra. More formally, we can construct \emph{$\Hcal$-classifying topoi} of $\T^{\Hcal}$-algebras, and characterise them by a Diaconescu-like equivalence. To this end, recall that for each doctrine $\T$:
\begin{itemize}
    \item[--]for a small $\T$-algebra $\mc A$, the algebra functor $a \colon \T \mc A \to \mc A$ which realises it is left adjoint to the unit $\eta_{\mc A} \colon \mc A \hook \T \mc A$, itself given by the corestriction of the Yoneda embedding;
    \item[--]topoi are always (large) $\T$-algebras.
\end{itemize}

\begin{defn}[Classifying topoi]\label{def:classifying-topoi}
    Let $\Hcal$ be a bounded logic and let $\mc A$ be a small $\T^{\Hcal}$-algebra, realised by a functor $a \colon \T^{\Hcal}(\mc A) \to \mc A$ in $\Lex$. The \emph{$\Hcal$-classifying topos} $\Cl^{\Hcal}(\mc A)$ of $\mc A$ is constructed as the topos of sheaves on the site $(\mc A, J^{\Hcal}_{\mc A})$, where $J^{\Hcal}_{\mc A}$ is the least topology localizing the components of the canonical natural transformation 
\[\begin{tikzcd}
	{\T^{\Hcal}(\mc A)} & {\Psh(\mc A)}
	\arrow[""{name=0, anchor=center, inner sep=0}, curve={height=-15pt}, hook, from=1-1, to=1-2]
	\arrow[""{name=1, anchor=center, inner sep=0}, "{\yo a }"', curve={height=15pt}, from=1-1, to=1-2]
	\arrow["{\delta}", between={0.2}{0.8}, Rightarrow, from=0, to=1]
\end{tikzcd}\]    
induced by the unit of the adjunction $a \dashv \eta_{\mc A}$. We note here that the topology $J^{\Hcal}_{\mc A}$ is \emph{subcanonical}, i.e.\ representable presheaves are sheaves.

This construction defines a 2-functor $\Cl^{\Hcal} \colon \Alg(\T^{\Hcal})\op \to \WRInj(\Hcal)$.
\end{defn}

\begin{thm}[Diaconescu's equivalence]\label{thm:diaconescu}
    Let $\Hcal$ be a bounded logic and let $\mc A$ be a small $\T^{\Hcal}$-algebra. For any topos $\mc E$ there is a pseudonatural equivalence of categories
    \[\Topoi(\mc E, \Cl^{\Hcal}(\mc A)) \simeq \Alg(\T^{\Hcal})(\mc A, \mc E).\]
\end{thm}

\begin{rem}[Topoi as classifying topoi, algebras as syntactic categories]
    Intuitively, \Cref{def:classifying-topoi} formalises the idea of classifying topoi for theories in the fragment of geometric logic represented by $\Hcal$ are formally in $\Hcal$. However, the question of whether every topos formally in $\Hcal$ is the $\Hcal$-classifying topos of some $\T^{\Hcal}$-algebra is still open.\footnote{To be more precise, whether every topos formally in $\Hcal$ is a \emph{coadjoint retract} of an $\Hcal$-classifying topos, see \cite[Lem.\ 1.1.5]{dilibertiLogicConcepts2category2025}.} Instead, we will see how the `symmetric' question, namely whether every $\T^{\Hcal}$-algebra is the syntactic category of a topos formally in $\Hcal$, precisely means that $\Hcal$ is \emph{conceptually complete} (cf \Cref{sec:everyalgisasyncat}). 
\end{rem}

\begin{rem}[$\Hcal$-presentations]\label{rem:classifying-topoi-have-presentations}
    We say that a topos $\Xcal$ formally in $\Hcal$ admits an \emph{$\Hcal$-presentation} if there is a geometric embedding $\mc X \hook \Psh(\mc C)$ in $\WRInj(\Hcal)$ into some presheaf topos. By \cite[Lem.\ 6.1.2]{dilibertiLogicConcepts2category2025}, the embedding $j \colon \Cl^{\Hcal}(\mc A) \hook \Psh(\mc A)$ is an $\Hcal$-presentation of $\Cl^{\Hcal}(\mc A)$ for each small $\T^\Hcal$-algebra $\mc A$. 
\end{rem}

\looseness=-1
We conclude this subsection with the most important definition for the sake of this paper, which will be the starting point for the next sections.

\begin{defn}[Conceptually complete logics]\label{def:cc-logic}
    A bounded logic $\Hcal$ is \emph{conceptually complete} if the 2-functor $\Cl^{\Hcal} \colon \Alg(\T^{\Hcal})\op \to \WRInj(\Hcal)$ is 2-fully-faithful.
\end{defn}

We shall wrap up this section by making a couple of final comments on the topic of conceptual completeness, borrowing the main table from \cite[\S 7.2]{dilibertiLogicConcepts2category2025}, where the authors speculate on the description of the ingredients of conceptual completeness for many concrete examples of logics. 

\begin{table}[!h]\renewcommand{\arraystretch}{1.3}
\begin{tabular}{|c|c|c|}
\hline
\textbf{Logic} --  \textbf{$\mathcal{H}$}  & $\mathsf{WRInj}_{\mr{pw}}(\mathcal{H})$  & $\mathsf{Alg}(\ms T^\mathcal{H}$) \\ \hline

geometric \hfill  $\mathcal{H}_{\text{eth}}$ & $\mathsf{Topoi}$  \hfill & \textsf{Pretopoi}$_\infty$ \hfill \\ \hline

integral \hfill  $\mathcal{H}_{\text{cl}}$ & $\mathsf{TocTopoi}$ \hfill  & \textsf{ConLex} \hfill \\ \hline

ess. alg. \hfill  $\mathcal{H}_{\text{all}}$ & $\mathsf{Free}^{*}$ \hfill  & \textsf{Lex} \hfill \\ \hline

ess. alg. w. $\bot$ \hfill  $\mathcal{H}_{\text{dom}}$ & $\mathsf{ClFree}^{*}$  \hfill & \textsf{siLex} \hfill \\ \hline

disjunctive \hfill  $\mathcal{H}_{\text{inn}}$ & $\mathsf{DisjTopoi}^{*}$ \hfill  & \textsf{Lexten}$_\infty$ \hfill \\ \hline

fin. disj. \hfill  $\mathcal{H}_{\text{pure}}$ & $\mathsf{fDisjTopoi}^{*}$ \hfill &  \textsf{Lexten} \hfill \\ \hline

regular \hfill   $\mathcal{H}_{\text{matte}}$ & $\mathsf{RegTopoi}^{*}$ \hfill  & \textsf{EffReg} \hfill \\ \hline

coherent \hfill $\mathcal{H}_{\text{flat}}$ & $\mathsf{CohTopoi}^{*}$ \hfill & \textsf{Pretopoi} \hfill \\ \hline

\end{tabular}
\label{tab:summary}
\end{table}

\looseness=-1
Besides the case of full geometric logic, which follows from \cite[Prop.\ 1.2.3]{dilibertiLogicConcepts2category2025}, and the case of coherent logic, which was deduced in \cite[\S 6.3]{dilibertiLogicConcepts2category2025} from Makkai's results in \cite{makkaiStoneDualityFirst1987}, the description of $\ms T^{\mc H}$-algebras for the remaining rows were conjectural at the time of that paper. In~\cref{sec:examples}, we will successfully provide a description for the algebras in the cases of essentially algebraic logic with falsum, finitary disjunctive logic, regular logic, and coherent logic --- with a new characterisation that does not rely on \cite{makkaiStoneDualityFirst1987}. Compared to the above table, the only case which we leave conjectural is thus that of disjunctive logic. We leave to further research, instead, a characterisation of the column $\WRInj_{\mr{pw}}(\Hcal)$\footnote{The subscript ``pw'' refers to \emph{pointwiseness} of the defining Kan extensions, which will not play any role in this paper (cf.\ \cite[\S 2]{dilibertiLogicConcepts2category2025}).}.

\begin{rem}[The role of effective quotients]
    It is well-known in the case of coherent and regular logic that coherent categories and regular categories alone cannot lead to a conceptual completeness result, due to the fact that a coherent or regular category has the same category of models of its exact completion (cf.~\cite[\S D1]{johnstoneSketchesElephantTopos2002}). Thus, obtaining pretopoi for coherent logic and (Barr-)exact categories for regular logic should be expected.
\end{rem}

\section{Conceptual completeness as `every algebra is a syntactic category'} \label{sec:everyalgisasyncat}

In this section, we exhibit an adjunction between the syntactic category and the classifying topos constructions associated to a logic $\Hcal$ (\Cref{prop:adjunction-syn-cl}). As a consequence, we will characterise \emph{conceptual completeness} for $\Hcal$ as the property that every $\T^{\Hcal}$-algebra can be recovered as the syntactic category of its classifying topos (\Cref{cor:cc-as-every-algebra-is-syntactic-category}). Throughout, fix a bounded logic $\Hcal$.

\subsection{Syntactic categories are $\T^\Hcal$-algebras}
We begin by showing how to equip syntactic categories with algebra structures: intuitively, the following proposition expresses how the syntactic category of a theory in a given fragment of geometric logic is endowed with the categorical structure prescribed by that fragment.

\begin{prop}[Syntactic categories are $\T^\Hcal$-algebras] \label{prop:syntactic-categories-are-algebras}
 For every topos $\Xcal$ formally in $\Hcal$, the syntactic category $\Syn^{\Hcal}(\Xcal)$ is a $\T^\Hcal$-algebra.
\end{prop}
\begin{proof}
Recall by \cite[Constr.\ 5.1.5]{dilibertiLogicConcepts2category2025} that the syntactic category $\Syn^{\Hcal}(\Xcal)$ can be canonically equipped with a topology $D^{\Hcal}_{\Xcal}$ defined as the pullback topology $i^*(J^{\operatorname{can}}_{\Xcal})$, where $i \colon \Syn^{\Hcal}(\Xcal) \hookrightarrow \Xcal$ is the inclusion and $J^{\operatorname{can}}_{\Xcal}$ is the canonical topology on $\Xcal$. Adapting the proof of \cite[Lem.\ C2.3.13]{johnstoneSketchesElephantTopos2002} to possibly-large base categories, we see that the topos of sheaves $\Sh(\Syn^{\Hcal}(\Xcal), D^{\Hcal}_{\Xcal})$ arises as the image factorisation of the geometric morphism $a_\Xcal \colon \Xcal \to \Psh(\Syn^{\Hcal}(\Xcal))$ induced by $i$:
\[\begin{tikzcd}[sep = small]
	\Xcal & {\Psh(\Xcal)} & {\Psh(\Syn^{\Hcal}(\Xcal))} \\
	& {\Sh(\Syn^{\mathcal H}(\mathcal X), D^{\mathcal H}_{\mathcal X})}
	\arrow[hook, from=1-1, to=1-2]
	\arrow["{a_{\Xcal}}"{description}, out=30, in=160, dashed, from=1-1, to=1-3]
	\arrow["{q_{\Xcal}}"', two heads, from=1-1, to=2-2]
	\arrow[from=1-2, to=1-3]
	\arrow[hook, from=2-2, to=1-3]
\end{tikzcd}\]
Note that $a_\Xcal$ is the geometric morphism corresponding to the inclusion $i$ via Diaconescu's equivalence (cf.\ \cite[Rem.\ 6.2.2]{dilibertiLogicConcepts2category2025}), i.e.\ $a_\Xcal^* \yo = i \colon \Syn^{\Hcal}(\Xcal) \hookrightarrow \Xcal$.
We now show that the morphism $a_\Xcal$ lies in $\WRInj(\Hcal) ( \Xcal, \Psh(\Syn^{\Hcal}(\Xcal)))$, so that by precomposition we obtain a structure functor $\T^\Hcal(\Syn^\Hcal(\Xcal)) \to \Syn^\Hcal(\Xcal)$. 
Consider the following situation, for any $f \colon \Ecal \to \Fcal$ in $\Hcal$ and any $x \colon \Ecal \to \Xcal$:
\[\begin{tikzcd}[column sep= 35pt]
	\Ecal & \Xcal \\
	\Fcal & {\Psh(\Syn^{\Hcal}(\Xcal))}
	\arrow["x", from=1-1, to=1-2]
	\arrow["f"', from=1-1, to=2-1]
	\arrow["{a_{\Xcal}}", from=1-2, to=2-2]
	\arrow["\ran_f x"{description}, dashed, from=2-1, to=1-2]
	\arrow["\ran_f(a_{\Xcal} x)"', dashed, from=2-1, to=2-2]
\end{tikzcd}\]
We need to show that $a_\Xcal \ran_f x \cong \ran_f(a_{\Xcal} x)$. Clearly this holds if and only if the two agree on representables, i.e.\ if and only if $\yo_y a_\Xcal \ran_f x \cong \yo_y \ran_f(a_{\Xcal} x)$ for any $y$ in $\Syn^{\Hcal}(\Xcal) = \WRInj(\Hcal)(\Xcal, \Set[\Obb])$.
\[\begin{tikzcd}[column sep = 35pt]
	\Ecal & \Xcal &[-10pt] {\Set[\Obb]} \\
	\Fcal & {\Psh(\Syn^{\Hcal}(\Xcal))}
	\arrow["x", from=1-1, to=1-2]
	\arrow["f"', from=1-1, to=2-1]
	\arrow["y", from=1-2, to=1-3]
	\arrow["{a_{\Xcal}}", from=1-2, to=2-2]
	\arrow["\ran_f x"{description}, dashed, from=2-1, to=1-2]
	\arrow["\ran_f(a_{\Xcal} x)"', dashed, from=2-1, to=2-2]
	\arrow["\yo_y"',bend right=30, from=2-2, to=1-3]
\end{tikzcd}\]
Note then that $\yo_y a_\Xcal  \cong y$ since $a_\Xcal$ is the geometric morphism corresponding to the inclusion $i$ via Diaconescu's equivalence. Hence, since $y$ lies in $\WRInj(\Hcal)$ and since $\yo_y$ preserves all right Kan extensions, we have:
\[ \yo_y a_\Xcal \ran_f x \cong y \ran_f x \cong \ran_f(y x ) \cong \ran_f(\yo_y a_{\Xcal} x ) \cong \yo_y \ran_f(a_{\Xcal} x).\]

Let now $\alpha_{\Xcal} \colon \T^\Hcal(\Syn^\Hcal(\Xcal)) \to \Syn^\Hcal(\Xcal)$ be given by precomposition with $a_\Xcal$. To see that it is a $\T^{\Hcal}$-algebra functor we need to show that $\alpha_{\Xcal} \dashv \eta$, where $\eta$ is the unit of the monad $\T^\Hcal$ at ${\Syn^{\Hcal}(\Xcal)}$ acting as the corestriction of the Yoneda embedding.
\begin{itemize}
    \item The unit of $\alpha_{\Xcal} \dashv \eta$ at an object $X \in \T^{\Hcal}(\Syn^\Hcal(\Xcal))$ is the natural transformation $\delta_X \colon X \Rightarrow \yo_{X \circ a_{\mathcal X}}$ given as follows. Writing $X$ as a small colimit of representable presheaves on $\Syn^{\Hcal}(\Xcal)$, we have 
    \[ a_{\mc X}^* (X) \cong a_{\mc X}^* (\colim_k \yo_{X_k}) \cong \colim_k (a_{\mc X}^* \yo_{X_k}) \cong \colim_k X_k,\]
    which, identifying $a_{\mc X}^*(X)$ with the morphism of $\Hcal$-injectives $X \circ a_{\mc X} \colon \Xcal \to \Set[\Obb]$, means that $\colim_k X_k$ exists in $\Syn^{\Hcal}(\Xcal)$. Thus, the family of coprojections $X_k \to \colim_{X_k}$ in $\Syn^{\Hcal}(\Xcal)$, identified with transformations between representables, induces a natural transformation 
    \[ X \cong \colim_k \yo_{X_k} \Rightarrow \yo_{\colim_k X_k} \cong \yo_{X\circ a_\Xcal}.\]

    \item The counit of $\alpha_{\Xcal} \dashv \eta$ at an object $y \in \Syn^{\Hcal}(\Xcal)$ is simply given by the isomorphism $\yo_{y} \circ a_{\Xcal} \cong y$.
\end{itemize}
We omit the routine verification of the triangle equalities.
\end{proof}

\begin{cor}[$\Syn^{\Hcal}$ lands in $\alg(\mathsf{T}^\Hcal)$]
    The 2-functor $\Syn^{\Hcal} \colon {\WRInj(\Hcal)}\op \to \lex$ can be upgraded to a 2-functor ${\WRInj(\Hcal)}\op \to {\alg(\mathsf{T}^\Hcal)}$. 
    \begin{proof}
        To prove that $\Syn^{\Hcal}$ also sends morphisms of $\Hcal$-injectives to $\T^{\Hcal}$-algebra morphisms, let $h \colon \Xcal \to \mc Y$ be a morphism in $\WRInj(\Hcal)$ and consider, for a functor $f \colon \mc A \to \Syn^{\Hcal}(\mc Y)$ in $\lex$, the diagram
\[\begin{tikzcd}[column sep = 40pt]
	{\mc A} & {\Syn^{\Hcal}(\mc Y)} \\
	{\T^{\Hcal}(\mc A)} & {\Syn^{\Hcal}(\mc X),}
	\arrow["f", from=1-1, to=1-2]
	\arrow["\eta"', from=1-1, to=2-1]
	\arrow["{\Syn^{\Hcal}(h)}", from=1-2, to=2-2]
	\arrow["{\alpha_{\mc Y}^{f}}"{description}, dashed, from=2-1, to=1-2]
	\arrow["{\alpha_{\mc X}^{\Syn^{\Hcal}(h)\circ f}}"', dashed, from=2-1, to=2-2]
\end{tikzcd}\]
        where $\alpha_{\mc Y}^f \coloneqq \alpha_{\mc Y} \circ \T^{\Hcal}(f)$ and $\alpha_{\mc X}^{\Syn^{\Hcal}(h)\circ f} \coloneqq \alpha_{\mc X} \circ \T^{\Hcal}(\Syn^{\Hcal}(h) \circ f)$ (see \cite[\S 5.3]{dilibertiLogicConcepts2category2025}). The functor $\Syn^{\Hcal}(h)$ is a $\T^{\Hcal}$-algebra morphism if the lower triangle in the above diagram commutes, for which it clearly suffices to show that $\Syn^{\mc H}(h) \circ \alpha_{\mc Y} \cong \alpha_{\mc X} \circ \T^{\Hcal}(\Syn^{\Hcal}(h))$, i.e.\ that the diagram
\[\begin{tikzcd}
	{\T^{\Hcal}\Syn^{\Hcal}(\mc Y)} & {\Syn^{\Hcal}(\mc Y)} \\
	{\T^{\Hcal}\Syn^{\Hcal}(\mc X)} & {\Syn^{\Hcal}(\mc X)}
	\arrow["{\alpha_{\mc Y}}", from=1-1, to=1-2]
	\arrow["{\T^{\Hcal}\Syn^{\Hcal}(h)}"', from=1-1, to=2-1]
	\arrow["{\Syn^{\Hcal}(h)}", from=1-2, to=2-2]
	\arrow["{\alpha_{\mc X}}"', from=2-1, to=2-2]
\end{tikzcd}\]
commutes. To see this, recall that $\alpha_{\mc Y}$ is given by precomposition with the geometric morphism $a_{\mc Y} \colon \mc Y \to \Psh(\Syn^{\Hcal}(\mc Y))$ corresponding, via Diaconescu's equivalence, to the inclusion $i_{\mc Y} \colon \Syn^{\Hcal}(\mc Y) \hook \mc Y$, and similarly for $\alpha_{\mc X}$. Hence, making the action of $\T^{\Hcal}$ on morphisms explicit, the above diagram commutes if and only if:
\[ h^* a^*_{\mc Y} (Y) \cong a^*_{\mc X} (\lan_{\Syn^{\Hcal}(h)\op }Y)\]
for any $Y \in \T^{\Hcal}(\Syn^{\Hcal}(\mc Y))$, identified with a presheaf on $\Syn^{\Hcal}(\mc Y)$. This latter isomorphism follows by writing $Y$ as a small colimit of representables and recalling that the functor $\Syn^{\Hcal}(h) \colon \Syn^{\Hcal}(\mc Y) \to \Syn^{\Hcal}(\mc X)$ is the restriction of $h^* \colon \mc Y \to \mc X$, i.e.\ $h^* \circ i_{\mc Y} \cong i_{\mc X} \circ \Syn^{\Hcal}(h)$.
    \end{proof}
\end{cor}

\begin{rem}[Two sites for the same topos?]\looseness=-1
At this point, it is worth reflecting on the following situation. For a topos $\Xcal$ formally in $\Hcal$, \Cref{prop:syntactic-categories-are-algebras} allows us to consider the classifying topos $\Cl^{\Hcal}(\Syn^{\Hcal}(\mc X))$, defined as recalled in \Cref{def:classifying-topoi} as the topos of sheaves on the site $(\Syn^{\Hcal} (\Xcal), J^{\Hcal}_{\Xcal})$. However, as recalled above, we can also consider the pullback of the canonical topology on $\Xcal$ along the embedding $\Syn^{\Hcal}(\mc X) \hookrightarrow \mc X$, which gives rise to the \emph{syntactic site} $(\Syn^{\Hcal}(\mc X), D^{\Hcal}_{\Xcal})$ for $\Xcal$. 

As we will see in \Cref{sec:cons-emb}, the matter of characterizing classifying topoi of syntactic categories as categories of sheaves on syntactic sites is tightly linked to conceptual completeness: more concretely, we will see how conceptual completeness of $\Hcal$ will imply that $\Cl^{\Hcal}(\Syn^{\Hcal}(\mc X)) \simeq \Sh ( \Syn^{\Hcal}(\mc X) , D^{\Hcal}_{\mc X}) $ for each topos $\Xcal = \Cl^{\Hcal}(\mc A)$ classifying a small $\T^{\Hcal}$-algebra $\mc A$.
\end{rem}

\subsection{A Diaconescu-like adjunction: $\Syn^{\Hcal}\dashv \Cl^{\Hcal}$}

We can now prove the main result of this section, namely an adjunction between the syntactic category and the classifying topos constructions (\Cref{prop:adjunction-syn-cl}). The proof of this fact will make use of the following general principle about injectivity classes.

\begin{lem}[Injectivity via fully-faithfulness]\label{lem:injectivity-via-embeddings}
Let $\mc M$ be a class of maps in a 2-category $\mc K$ and let $\mc{S}$ be any sub-2-category of $\mc K$ defined by (weak) injectivity with respect to $\mc M$. Let $f \colon X \to Y$ be any morphism in $\mc K$ and let $i \colon Y \hook Z$ be a fully faithful morphism lying in $\mc S$. Then, $f \in \mc S$ if and only if $i\circ f \in \mc S$.
\begin{proof}
    Clearly $i\circ f \in \mathcal S$ if $f \in \mathcal S$; suppose conversely that $i\circ f \in \mathcal S$ and fix $\mathcal S \coloneqq \WRInj(\mc M)$ for concreteness. Consider the following situation, for any $x \colon A \to X$ with $X \in \mathcal S$ and any $g \colon A \to B$ in $\mc M$:
\[\begin{tikzcd}[column sep = 40pt]
	A & X  \\
	B & {Y}
	\arrow["x", from=1-1, to=1-2]
	\arrow["g"', from=1-1, to=2-1]
	\arrow["{f}", from=1-2, to=2-2]
	\arrow["\ran_g x"{description}, dashed, from=2-1, to=1-2]
	\arrow["\ran_g (f x)"', dashed, from=2-1, to=2-2]
\end{tikzcd}\]
Since $i\circ f$ and $i$ both lie in $\mathcal S$, we have:
\[ i f \ran_g x \cong \ran_g (i f x) \cong i \ran_g (f x )\]
from which $f \ran_g x \cong \ran_g (fx)$ since $i$ is fully faithful, meaning that $f \in \mathcal S$.

\end{proof}
\end{lem}

\begin{prop}[$\Syn^{\Hcal}\dashv \Cl^{\Hcal}$] \label{prop:adjunction-syn-cl}
There is a 2-adjunction:
\[\begin{tikzcd}
	{\WRInj(\Hcal)} && {{\alg(\T^{\Hcal})}\op}
	\arrow[""{name=0, anchor=center, inner sep=0}, "{\Syn^{\Hcal}}", curve={height=-12pt}, from=1-1, to=1-3]
	\arrow[""{name=1, anchor=center, inner sep=0}, "\Cl^{\Hcal}", curve={height=-12pt}, from=1-3, to=1-1]
	\arrow["\dashv"{anchor=center, rotate=-90}, draw=none, from=0, to=1]
\end{tikzcd}\]
\end{prop}
\begin{proof}
By Diaconescu's equivalence (\Cref{thm:diaconescu}) we know that there is an equivalence $\Alg(\T^{\Hcal})(\mc A, \Ecal) \simeq \mathsf{Topoi}(\Ecal, \Cl^{\Hcal}(\mc A))$ for each small $\T^{\Hcal}$-algebra $\mc A$ and each topos $\Ecal$. Therefore, to prove that $\Syn^{\Hcal} \dashv \Cl^{\Hcal}$, we now show that a $\T^{\Hcal}$-algebra morphism $F\colon \mc A \to \Ecal$ lands in the syntactic category $\Syn^{\Hcal}(\mc E) \subseteq \mc E$ if and only if its corresponding geometric morphism $f \colon \mc E \to \Cl^{\Hcal}(\mc A)$ lies in $\WRInj(\Hcal)$.

\begin{itemize}
    \item Assume first that $F \colon \mc A \to \Ecal$ factors through $\Syn^{\Hcal}(\mc E)$, and denote by $\bar F \colon \mc A \to \Syn^{\Hcal}(\mc E)$ its corestriction. Consider the geometric morphism $\mc P \bar F \colon \Psh(\Syn^{\Hcal}(\mc E)) \to \Psh(\mc A)$ induced by $\bar F$, whose inverse image acts by left Kan extension along $\bar F$, and note that the diagram below commutes.
\[\begin{tikzcd}[ampersand replacement=\&]
	\Ecal \& {\Cl^{\Hcal}(\mc A)} \& {\psh(\mc A)} \\
	\& {\psh(\Syn^{\Hcal}(\Ecal))}
	\arrow["f", from=1-1, to=1-2]
	\arrow["a_{\mc E}"', from=1-1, to=2-2]
	\arrow[hook, from=1-2, to=1-3]
	\arrow["\mc P \bar F"', from=2-2, to=1-3]
\end{tikzcd}\]
This follows since, recalling that the topology defining $\Cl^{\Hcal}(\mc A)$ is subcanonical, both inverse image functors $\Psh(\mc A) \to \mc E$ act as $F$ on representables. 

Therefore, the composite $\Ecal \xrightarrow{f} \Cl^{\Hcal}(\mc A) \hook \psh(\mc A)$ is a morphism of $\Hcal$-injectives by composition, since $a_{\mc E}$ is one by \Cref{prop:syntactic-categories-are-algebras} and $\mc P \bar F$ preserves all right Kan extensions as it is a right adjoint in $\Topoi$ (cf.\ \cite[Term.\ 5.2.4]{dilibertiLogicConcepts2category2025}); in particular, $f\colon \Ecal \to \Cl^{\Hcal}(\mc A)$ is thus a morphism of $\Hcal$-injectives by \Cref{lem:injectivity-via-embeddings} since the embedding $\Cl^{\Hcal}(\mc A) \hook \Psh(\mc A)$ is an $\Hcal$-presentation (cf.\ \Cref{rem:classifying-topoi-have-presentations}).

\item Conversely, suppose that $f\colon \Ecal \to \Cl^{\Hcal}(\mc A)$ is a morphism of $\Hcal$-injectives. Note that, for each object $a \in \mc A$, the representable $\yo_a \colon \Cl^{\Hcal}(\mc A) \to \Set[\Obb]$ is a morphism of $\Hcal$-injectives. This follows since --- again, the topology defining $\Cl^{\Hcal}(\mc A)$ being subcanonical --- it factors as the composite of the $\Hcal$-presentation $\Cl^{\Hcal}(\mc A) \hook \Psh(\mc A)$ with a representable presheaf, which is a right adjoint in $\Topoi$. Hence, the composite $\yo_a \circ f \colon \mc E \to \Set[\Obb]$ is also a morphism of $\Hcal$-injectives, meaning that $F(a)$ lies in $\Syn^{\Hcal}(\mc E) $.
\end{itemize}
\end{proof}

\looseness=-1
From the previous proof we can extract the counit of the adjunction $\Syn^{\Hcal}\dashv \Cl^{\Hcal}$, which is given at a small $\T^{\Hcal}$-algebra $\mc A$ by the corestriction $i_\mathcal{A} \colon \mc A \hook \Syn^{\Hcal} (\Cl^{\Hcal} (\mc A))$ of the Yoneda embedding. Via the 2-adjunction $\Syn^{\Hcal}\dashv \Cl^{\Hcal}$, we can thus equivalently characterise conceptual completeness of $\Hcal$ in the sense of \Cref{def:cc-logic} --- that is, the right adjoint being 2-fully-faithful --- as the property that every $\T^{\Hcal}$-algebra is a syntactic category.

\begin{cor}[Conceptual completeness as ``every algebra is a syntactic category''] \label{cor:cc-as-every-algebra-is-syntactic-category}
$\Hcal$ is conceptually complete if and only if, for every small $\T^{\Hcal}$-algebra $\mathcal A$, the embedding $i_\mathcal{A} \colon \mc A \hookrightarrow \Syn^{\Hcal} (\Cl^{\Hcal} (\mc A))$ is an equivalence of categories.
\end{cor}

\section{Four easy pieces} \label{sec:examples}
\looseness=-1
In this section, we show that the logics described in \Cref{ssec:logics-and-doctrines}, namely \emph{coherent logic} $\Hflat$, \emph{regular logic} $\Hmatte$, \emph{essentially algebraic logic with falsum} $\Hdom$, and \emph{finitary disjunctive logic} $\Hpure$, are conceptually complete. Before diving into each individual case, let us describe the general approach we will follow in the proofs of this section.

The main difficulty in proving conceptual completeness results in our framework is that, for a general logic $\Hcal$, we do not have an explicit description of a number of intertwined entities: the doctrine $\ms T^{\mc H}$, the morphisms of $\Hcal$-injectives, and the $\T^{\Hcal}$-algebras. For concreteness, the formula for $ \mathsf{T}^{\Hflat}$ is given by \[\mathsf{T}^{\Hflat}(\mc C) = \Syn^{\Hflat}(\Psh\mc C) =  \WRInj(\Hflat)(\psh\mc C, \Set[\mathbb{O}]),\]
but since we do not have a full description of the morphisms of $\Hflat$-injectives, it is quite hard to guess what presheaves will end up populating the syntactic category of $\psh (\mc C)$ --- in this case, the \emph{coherent} ones. Clearly, this also means that we do not know a priori what kind of colimits our algebras will admit, as these coincide with those prescribed by the monad of $\Hflat$-syntax --- in this case, finite disjoint coproducts and effective quotients. All these questions capture the essence of proving conceptual completeness: precisely in order to answer them, in \cite{dilibertiLogicConcepts2category2025}, the authors invoke Makkai's version of conceptual completeness (cf.\ the proofs of \cite[Thm.\ 6.3.2 and Cor.\ 6.3.4]{dilibertiLogicConcepts2category2025} and the discussion in \cite[Rem.\ 6.3.6]{dilibertiLogicConcepts2category2025}). 

The key idea underlying this section is that, for each of the four logics $\Hcal$ above, even though {a priori} we cannot fully characterise the doctrine $\ms T^{\mc H}$, we can still take an \emph{educated guess} for it. Indeed, in each case we can borrow a doctrine $\T$ from the literature on categorical logic which is characterised by the property that $\Hcal$ coincides with the class of morphisms whose direct image is a $\T$-algebra morphism: for the necessary background, we refer the reader to, e.g., \cite{garnerLexColimits2012,tendasFlatnessWeaklyLex2024}, or the discussion in \cite[\S 6.4]{dilibertiBiaccessibleBipresentable2Categories2024}). Concretely, this means that we take $\T$ to be given by:
\begin{itemize}
    \item[($\Hflat$)] the \emph{free pretopos} construction;
    \item[($\Hmatte$)] the \emph{free (Barr-)exact category} construction;
    \item[($\Hdom$)] the \emph{free lex category with strict initial object} construction;
    \item[($\Hpure$)] the \emph{free lextensive category} construction.
\end{itemize}
\looseness=-1
In each case, we are able to show that the usual `classifying topos' construction for $\ms T$-algebras lands in $\mc H$-injectives (cf.~\cref{lem:representables-lie-in-syntactic-category}). Together with the \emph{reduction lemma} below (\Cref{prop:generic-cc}), this will allow us to derive conceptual completeness for $\Hcal$.

\subsection{The reduction lemma}First note that, in each of the four cases above, small $\ms T$-algebras admit a
`classifying topos' construction in the traditional sense of categorical logic: that is, for any small $\ms T$-algebra $\mc A$, there is a subcanonical topology $J$ on $\mc A$ such that for any topos $\mc E$:
\[ \Alg(\ms T)(\mc A,\mc E) \simeq \Topoi(\mc E,\sh(\mc A,J)). \]
Concretely, the four topologies are given by the coherent topology on a pretopos, the regular topology on an exact category, the topology having an empty covering on the strict initial object for a lex category with a strict initial object, and the extensive topology on a lextensive category. 

We now show that these topoi formally belong to the respective logic $\Hcal$, and that it suffices to test conceptual completeness for $\Hcal$ against them: as emphasised already in \cite[Rem.\ 5.1.2]{dilibertiLogicConcepts2category2025}, we will then see how the crucial step to achieve conceptual completeness is to characterise their syntactic categories.

\begin{lem}\label{lem:representables-lie-in-syntactic-category}
For any small $\T$-algebra $\mc A$, the inclusion $j \colon \Sh(\mc A, J) \hook \Psh(\mc A)$ is an $\Hcal$-presentation of $\sh(\mc A, J)$. In particular, $\sh(\mc A, J)$ lies in $\WRInj(\Hcal)$. 

Moreover, there is a canonical inclusion:
\[ \mc A \hookrightarrow \Syn^{\Hcal}(\sh(\mc A, J)).\]

\end{lem}
\begin{proof}
    The first claim follows from~\cite[Prop.\ 4.0.1]{dilibertiLogicConcepts2category2025} since, by assumption, direct images of geometric morphisms in $\Hcal$ are $\T$-algebra morphisms. To see that $\mc A$ is contained in the $\mc H$-syntactic category of $\sh(\mc A, J)$, for any $a \in \mc A$ and any $f \colon \mc E \to \mc F$ in $\Hcal$, consider the following diagram:
\[\begin{tikzcd}[sep = 30pt]
	{\mc E} & {\sh(\mc A, J)} \\
	& {\psh(\mc A)} \\
	{\mc F} & {\Set[\mbb O]}
	\arrow["x", from=1-1, to=1-2]
	\arrow["f"', from=1-1, to=3-1]
	\arrow["j", hook, from=1-2, to=2-2]
	\arrow["{\yo_a}", from=2-2, to=3-2]
	\arrow["{{\ran_{f}x}}"{description}, curve={height=-6pt}, dashed, from=3-1, to=1-2]
	\arrow["{{\ran_f(jx)}}"{description}, dashed, from=3-1, to=2-2]
	\arrow["{{\ran_f(\yo_a jx)}}"', dashed, from=3-1, to=3-2]
\end{tikzcd}\]
\looseness=-1
Note that, since the topology $J$ is subcanonical, the composite $\yo_a j$ can be identified with the representable $\yo_a \colon \sh(\mc A, J) \to \Set[\Obb]$ itself. Then, since $j$ preserves right Kan extensions along $f$ as it is an $\Hcal$-presentation of $\sh(\mc A, J)$, and since representable presheaves preserve all right Kan extensions as they are right adjoint in $\Topoi$, we have that 
\[ \yo_a \circ \ran_f x \cong \ran_f (\yo_a \circ x) .\]
Thus, $\yo_a$ lies in $\Syn^{\Hcal}(\sh(\mc A, J))\coloneqq \WRInj(\Hcal)(\sh(\mc A, J), \Set[\Obb])$, and hence the Yoneda embedding corestricts to an inclusion $\mc A \hook \Syn^{\Hcal}(\sh(\mc A, J))$.
\end{proof}

\begin{prop}[Reduction lemma]\label{prop:generic-cc}
    If the inclusion $\mc A \hook \Syn^{\mc H}(\sh(\mc A, J))$ is an equivalence for any small $\ms T$-algebra $\mc A$, then $\ms T^{\mc H} = \ms T$, and $\mc H$ is conceptually complete.
\end{prop}
\begin{proof}
    Note first that, for any category $\mc C$ in $\lex$,
    $ \sh(\ms T\mc C, J) \simeq \psh(\mc C), $
    since by the usual Diaconescu's theorem (e.g.,\ \cite[Thm.\ 7.7.2]{maclaneSheavesGeometryLogic1994}) we have 
    \[\Topoi(\mc E ,  \sh(\ms T\mc C, J)) \simeq \Alg(\T)(\T \mc C, \mc E) \simeq \LEX(\mc C, \mc E) \simeq \Topoi(\mc E, \Psh(\mc C))\]
    for any topos $\mc E$. Hence, by hypothesis, we have an equivalence:
    \[ \ms T(\mc C) \simeq \Syn^{\mc H}(\sh(\T \mc C, J)) \simeq \Syn^{\mc H}(\psh(\mc C)) = \ms T^{\mc H}(\mc C), \]
    which implies that the doctrine $\ms T^{\mc H}$ coincides with $\ms T$. In particular, for any small $\ms T$-algebra $\mc A$, the $\mc H$-classifying topos $\Cl^{\mc H}(\mc A)$ is equivalent to $\sh(\mc A, J)$, since by Diaconescu's equivalence (\Cref{thm:diaconescu}) we have that
    \[\Topoi(\mc E,\Cl^{\mc H}(\mc A)) \simeq  \Alg(\ms T^{\mc H})(\mc A,\mc E)  = \Alg(\T)(\mc A , \mc E)\simeq \Topoi(\mc E, \Sh(\mc A, J)) \]
    for any topos $\mc E$.  Thus, the hypothesis also entails that the inclusion $ \mc A \hook \Syn^{\mc H}\Cl^{\mc H}(\mc A)$
    is an equivalence for any small $\ms T$-algebra $\mc A$, which by \Cref{cor:cc-as-every-algebra-is-syntactic-category} means that $\mc H$ is conceptually complete. 
\end{proof}

\begin{rem}\label{rem:clever-sublogics}
    In fact, in the proofs below, we will single out convenient subclasses of maps $\mc H' \subseteq \mc H$, and show that for any small $\ms T$-algebra $\mc A$, the inclusion
\[ \mc A \hook \Syn^{\mc H'}(\sh(\mc A, J)) \]
is an equivalence. This clearly entails that $\mc A \simeq \Syn^{\mc H}(\sh(\mc A, J))$, since by $\mc H' \subseteq \mc H$ we know that the above inclusion factors through $\Syn^{\mc H}(\sh(\mc A, J))$. We point out here that the proofs in this section have a somewhat distinct flavour from the rest of this paper: in particular, the choices for the subclasses $\mc H'$ are quite \emph{specific} in each case, as we still lack a modular treatment of all logics. 
\end{rem}

\subsection{Coherent logic}\label{ssec:cc-coherent-logic}We begin by considering the case of \emph{coherent logic} $\Hflat$ (\Cref{def:coherent-logic}). Following the general strategy we described earlier, consider the \emph{free pretopos} doctrine $\T^{\mr{pretop}}$: since its algebras are pretopoi, note that $\Hflat$ coincides with the class of geometric morphisms whose direct images are $\T^{\mr{pretop}}$-algebra functors. Concretely, $\T^{\mr{pretop}}$ is defined by mapping a category $\mc C$ in $\lex$ to the full subcategory of $\Psh(\mc C)$ spanned by the \emph{coherent objects} (cf.\ \cite{garnerLexColimits2012}), whose definition we now recall from \cite[D3.3]{johnstoneSketchesElephantTopos2002}.

\begin{defn}[Compact objects]\label{def:compact-objects}
    An object $X$ in a topos $\mc E$ is \emph{compact} if the frame $\Sub_{\mc E}(X)$ is compact, i.e.\ if every union of subobjects of $X$ covering it admits a finite subcovering, and it is \emph{coherent} if moreover, for any pair of arrows $f\colon Y \to X$ and $g \colon Z \to X$ with $Y$ and $Z$ compact, the pullback $Y \times_X Z$ is compact.
\end{defn}

Towards this proof, we will rely on the subclass $\mc H_{\beta} \subseteq \mc H_{\mr{flat}}$ of \emph{$\beta$-maps} already considered in \cite{dilibertiGeometryCoherentTopoi2022, dilibertiLogicConcepts2category2025}, i.e.\ the geometric morphisms $i_I \colon \Set^I \hook \Sh(\beta I)$ induced, for each set $I$ as a discrete topological space, by the inclusion $i_I \colon I \hook \beta I$ of $I$ into its Stone-Čech compactification.

Fix from now on a \emph{coherent topos} $\Xcal$, that is, a topos of the form $\Sh(\mc C, J_{\mr{coh}})$ where $\mc C$ is a small pretopos and $J_{\mr{coh}}$ is the \emph{coherent coverage} on $\mc C$; for these definitions, we refer the reader to \cite[\S A1.4 and Ex.\ A2.1.11]{johnstoneSketchesElephantTopos2002}. By \Cref{lem:representables-lie-in-syntactic-category}, the inclusion $\Xcal \hook \Psh(\mc C)$ is an $\Hflat$-presentation of $\Xcal$, which thus lies in $\WRInj(\Hflat)$, and hence in $\WRInj(\Hcal_{\beta})$. We will characterise $\Syn^{\Hcal_\beta}(\Xcal)$ as the full subcategory of $\Xcal$ spanned by the coherent objects (see also \cite[Thm.\ 4.4.3]{dilibertiLogicConcepts2category2025}). This will entail that the inclusion $\mc C \hook \Syn^{\Hcal_\beta}(\Xcal)$ is an equivalence, and hence that $\Hcal_\beta$ and $\Hflat$ are conceptually complete by the reduction lemma (\Cref{prop:generic-cc}).

Based on the following observation, to show that all objects of $\Syn^{\Hcal_\beta}(\Xcal)$ are coherent it suffices to prove that all objects of $\Syn^{\Hcal_\beta}(\Xcal)$ are compact.

\begin{lem}\label{lem:syntactic-objects-are-compact-iff-coherent}
    The following are equivalent:
   \begin{enumerate}
       \item every object of $\Syn^{\Hcal_\beta}(\Xcal)$ is compact in $\Xcal$;
       \item every object of $\Syn^{\Hcal_\beta}(\Xcal)$ is coherent in $\Xcal$.
   \end{enumerate}
\begin{proof}
\looseness=-1
Clearly, it suffices to show that $(1)\Rightarrow(2)$, so let $X \in \Syn^{\Hcal_{\beta}}(\Xcal)$; to show that it is coherent, since coherent objects of $\Xcal$ determine a generating subcategory of $\Xcal$ (see, e.g., \cite[Thm.\ D3.3.7]{johnstoneSketchesElephantTopos2002}), we can reduce to show that the pullback $Y \times_X Z$ is compact for any pair of arrows $f\colon Y \to X$ and $g \colon Z \to X$ with $Y$ and $Z$ coherent. But in that case, $Y$ and $Z$ lie in $\Syn^{\Hcal_\beta}(\mc X)$ by~\cref{lem:representables-lie-in-syntactic-category}, hence so do $f$ and $g$ since the inclusion $\Syn^{\Hcal_\beta}(\mc X) \hook \Xcal$ is full, and hence so does $Y \times_X Z$ since $\Syn^{\Hcal_\beta}(\mc X)$ is closed under finite limits in $\Xcal$; in particular, by assumption, $Y\times_X Z$ is compact.
    \end{proof}
\end{lem}

We thus proceed to characterise compact objects in $\Xcal$ in terms of the $\beta$-maps.

\begin{prop}\label{prop:compact-objects-in-coherent-topos}
    An object $X \in \Xcal$ is compact if and only if, for each set $I$ and each geometric morphism $x \colon \Set^I \to \Xcal$, the canonical comparison 2-cell below $\sigma^X \colon X \circ \ran_{i_I} x \Rightarrow \ran_{i_I} (X \circ x)$ is an epimorphism
    \[\begin{tikzcd}[column sep = 35pt, row sep=25pt]
	{\Set^I} & {\mc X} \\
	{\sh(\beta I)} & {\Set[\mbb O]}
	\arrow["x", from=1-1, to=1-2]
	\arrow["{i_I}"', hook, from=1-1, to=2-1]
	\arrow["X", from=1-2, to=2-2]
	\arrow[""{name=0, anchor=center, inner sep=0}, "{{\ran_{i_I} x}}"{description}, dashed, from=2-1, to=1-2]
	\arrow[""{name=1, anchor=center, inner sep=0}, "{{\ran_{I_i} (X\circ x)}}"', dashed, from=2-1, to=2-2]
	\arrow["\sigma^X"{pos=0.3}, between={0.2}{0.8}, Rightarrow, from=0, to=1]
    \end{tikzcd}\]
\begin{proof}
    First, suppose $X$ is compact. Then, there exists a cover $e \colon Y \surj X$ where $Y$ is coherent (this also follows, e.g., from \cite[Thm.\ D3.3.7]{johnstoneSketchesElephantTopos2002}). Computing explicitly $(\ran_{i_I}(X\circ x))^*$ as $\lan_{\yo^{\Obb}}((i_I)_*x^*X^*)$, we can consider the following commuting diagram in $\Sh(\beta I)$:
\[\begin{tikzcd}
	{(\ran_{i_I}x)^*Y} & {(i_I)_*x^*(Y)} \\
	{(\ran_{i_I}x)^*X} & {(i_I)_*x^*(X)}
	\arrow["\sigma^Y", from=1-1, to=1-2]
	\arrow[from=1-1, to=2-1]
	\arrow[from=1-2, to=2-2]
	\arrow["\sigma^X"', from=2-1, to=2-2]
\end{tikzcd}\]
Note then that the vertical arrows are epimorphisms since so is $e$ and both $(\ran_{i_I}x)^*$ and $(i_I)_*x^*$ preserve epimorphisms; moreover, the top arrow is an isomorphism since $Y$ is coherent and hence it lies in $\Syn^{\Hcal_\beta}(\Xcal)$ by \Cref{lem:representables-lie-in-syntactic-category}. Therefore, the bottom arrow, corresponding to $\sigma^X \colon X \circ \ran_{i_I} x \Rightarrow \ran_{i_I} (X \circ x)$, is an epimorphism.

Conversely, suppose that $X$ is not compact. Let $\set{ X_i }_{i\in I}$ be the lattice of compact subobjects of $X$, which is a directed poset as compact objects are closed under finite coproducts and quotients. By assumption, none of the inclusions $X_i \hook X$ is an isomorphism; thus, since $\Xcal$ has enough points by Deligne's theorem \cite[Thm.\ D3.3.13]{johnstoneSketchesElephantTopos2002}, for each $i\in I$ we can choose a point $x_i \colon \Set \to \Xcal$ and an element $t_i \in x^*_i(X) \setminus x^*_i(X_i)$. Let then $\mu\in\beta I$ be an ultrafilter such that $\dv i \in \mu$ for each $i\in I$, which exists since $I$ is directed. Seeing $\mu$ as a point of $\Sh(\beta I)$, and calling $x \colon \Set^I \to \Xcal$ the geometric morphism induced by the family $\set{x_i}_{i\in I}$, we obtain the following diagram.
    \[\begin{tikzcd}[column sep = 35pt, row sep=25pt]
	& {\Set^I} & {\mc X} \\
	\Set & {\sh(\beta I)} & {\Set[\mbb O]}
	\arrow["x", from=1-2, to=1-3]
	\arrow["{{i_I}}"', hook, from=1-2, to=2-2]
	\arrow["X", from=1-3, to=2-3]
	\arrow["\mu"', from=2-1, to=2-2]
	\arrow[""{name=0, anchor=center, inner sep=0}, "{{{\ran_{i_I} x}}}"{description}, dashed, from=2-2, to=1-3]
	\arrow[""{name=1, anchor=center, inner sep=0}, "{{{\ran_{I_i} (X\circ x)}}}"', dashed, from=2-2, to=2-3]
	\arrow["\sigma^X"{pos=0.3}, between={0.2}{0.8}, Rightarrow, from=0, to=1]
    \end{tikzcd}\]

Towards a contradiction, suppose that $\sigma^X$ is an epimorphism. Then, in particular we have a surjection:
\[ \sigma_\mu^X \colon \mu^* (\ran_{i_I} x)^* (X) \surj \mu^*( i_I)_* x^*(X).\]
Write $X$ as a canonical colimit $X \cong \colim_{A \in \Xcal_{\mr{coh}}/X}A$ indexed by its category of elements over $\mc X_{\mr{coh}}$. Therefore, by the explicit computations of right Kan extensions for coherent topoi of \cite{dilibertiGeometryCoherentTopoi2022}, and since $\mu^*i_I$ is isomorphic to the ultraproduct functor $\int(-)d\mu$ as shown \emph{ibid.}, we can identify $\sigma_\mu^X$ with a surjection:
\[ \colim_{A \in \Xcal_{\mr{coh}}/X} \int x^*_i A d \mu \surj \int x^*_i (X) d \mu .\]

\looseness=-1
Note then that the family $(t_i)_{i\in I}$ defines an element of the ultraproduct in the codomain, so that by surjectivity there exists an arrow $a \colon A \to X$ from a coherent object $A$ and an element $u \in \int x_i^* A d \mu $ that maps to $(t_i)_{i\in I}$. Concretely, $u$ is represented by a family of elements $(u_i \in x^*_i(A))_{i\in U}$ for some $U \in \mu$, such that for all $i\in U$:
\[ x_i^*(a) (u_i) = t_i .\]
However, since $A$ is coherent, the image of $a \colon A \to X$ is a compact subobject of $X$, meaning that there exists some $i_0 \in I$ such that $X_{i_0}$ is the image of $a$, and hence $\dv i_0 \in \mu$. This implies that $(\dv i_0) \cap U \in \mu$, so that there exists some compact subobject $X_i \hook X$ containing $X_{i_0}$ and such that $t_i$ lies in the image of $x_i^*(a)$, namely $x_i^*(X_{i_0})$, which is absurd as $ x_i^*(X_{i_0})\subseteq x_i^*(X_i)$.
\end{proof}
\end{prop}

\begin{rem}
    The previous proposition can be seen as a rephrasing of \cite[Lem.\ 2.3.6]{lurieUltracategories2018} in our framework. Indeed, thinking of the objects of $\Xcal$ --- in hindsight --- as a \emph{left ultrafunctor} on the ultracategory of points of $\Xcal$, then we can conceptualise \Cref{prop:compact-objects-in-coherent-topos} as the crucial step towards identifying the \emph{ultrafunctors} with the coherent objects.
\end{rem}

\begin{thm}[Coherent logic is conceptually complete]\label{flatcc}
    The doctrine $\ms T^{\mc H_\beta}$ coincides with the free pretopos construction over $\lex$, and $\mc H_{\beta}$ is conceptually complete; the same holds for $\mc H_{\mr{flat}}$.
\end{thm}
\begin{proof}
    By \Cref{prop:compact-objects-in-coherent-topos} it follows that every object $X \in \Syn^{\mc H_{\beta}}(\mc X)$ is compact. Thus, by \Cref{lem:syntactic-objects-are-compact-iff-coherent}, it follows that every object in $\Syn^{\mc H_{\beta}}(\mc X)$ is coherent, meaning that the inclusion $\mc A \hook \Syn^{\Hcal_{\beta}}(\sh (\mc A, J_{\mr{coh}}))$ is an equivalence for each small pretopos $\mc A$ (e.g.,\ by \cite[Thm.\ D3.3.7]{johnstoneSketchesElephantTopos2002}). The rest follows then by \Cref{prop:generic-cc} and \Cref{rem:clever-sublogics}.
\end{proof}

\subsection{Regular logic}\label{ssec:cc-regular-logic}We proceed to the case of \emph{regular logic} $\Hmatte$ (\Cref{def:regular-logic}). This proof will proceed in formal analogy with the coherent case, starting from the \emph{free exact category} doctrine $\T^{\mr {ex}}$: since its algebras are exact categories, in this case too we have that $\Hmatte$ coincides with the class of geometric morphisms whose direct images are $\T^{\mr{ex}}$-algebra functors. Concretely, $\T^{\mr{ex}}$ is defined by mapping a category $\mc C$ in $\lex$ to the full subcategory of $\Psh(\mc C)$ spanned by the \emph{regular objects} (cf.~\cite{carboniRegularExactCompletions1998}), whose definition we now recall.

\begin{defn}[\protect{\cite[Def.\ 2.1.14]{caramelloTheoriesSitesToposes2018}}]
    An object $X$ in a topos $\mc E$ is \emph{supercompact} if $X$ is completely join-irreducible in the frame $\Sub_{\mc E}(X)$, and it is \emph{regular} if moreover, for any pair of arrows $f \colon Y \to X$ and $g \colon Z \to X$ with $Y$ and $Z$ supercompact, the pullback $Y \times_X Z$ is supercompact.
\end{defn}

Towards this proof too, we will single out a convenient sublogic $\Hcal_\alpha \subseteq \Hmatte$ for which we are able to compute (enough) syntactic categories explicitly. For each set $I$, let $\alpha I$ be its one-point compactification, that is, the topological space whose underlying set is the disjoint union $I\cup\set{\infty}$ of $I$ with an additional point, and whose open subsets are:
\begin{itemize}
    \item any subset of $I$, and
    \item any subset of the form $U \cup \{\infty\}$ where $U\in\mc P^\omega(I)$, where $\mc P^\omega(I)$ is the set of cofinite subsets of $I$.
\end{itemize}
We denote by $i_I\colon \Set^I \hook \Sh(\alpha I)$ the geometric morphism induced by the inclusion $i_I \colon I \hook \alpha I$. We refer to these geometric morphisms as the \emph{$\alpha$-maps}, and we denote by $\Hcal_\alpha$ the logic they define.

\begin{lem}
    $\Hcal_\alpha \subseteq \Hmatte$.
\begin{proof}
    Concretely, we need to show that, for each set $I$, the geometric morphism $i_I \colon \Set^I \hook \Sh(\alpha I)$ is matte, i.e.\ its direct image preserves epimorphisms; let then $f \colon Y \surj X$ be an epimorphism in $\Set^I$. To see that $(i_I)_*f$ is an epimorphism in $\Sh(\alpha I)$, it suffices to see that it is stalkwise surjective (see, e.g., \cite[Prop.\ II.6.6]{maclaneSheavesGeometryLogic1994}). Note then that, since every singleton $\{i\} \subseteq I$ is discrete in $\alpha I$, stalks of $(i_I)_*f$ at each $i \in I \subseteq \alpha I$ are simply given by the components $f_i \colon Y_i \surj X_i$, which are clearly surjective; thus, we only have to consider the stalk at $\infty \in \alpha I$, explicitly given by the universal map
    \[ ((i_I)_*Y)_\infty \cong \colim_{U \in \mc P^\omega(I) } Y(U) \to \colim_{U\in \mc P^\omega(I)} X(U) \cong ((i_I)_*X)_\infty.\]
    To see that the latter is surjective, note that each component map 
    \[ Y(U) \cong \prod_{i\in U}Y_i \surj \prod_{i\in U}X_i \cong X(U) \]
    is a surjection\footnote{Technically, using the Axiom of Choice in $\Set$.}; thus, the claim follows since $\mathcal P^\omega(I)$ is filtered and surjections are closed under filtered colimits.
\end{proof}
\end{lem}

Fix from now on a \emph{regular topos} $\Xcal$, that is, a topos of the form $\Sh(\mc C, J)$ where $\mc C$ is an \emph{exact category} and $J_{\mr{reg}}$ is the \emph{regular coverage}; for these definitions, we refer the reader to \cite[\S A1.3 and Ex.\ A2.1.11]{johnstoneSketchesElephantTopos2002}. Again by \Cref{lem:representables-lie-in-syntactic-category} we have that $\Xcal$ lies in $\WRInj(\Hmatte)$, and hence in $\WRInj(\Hcal_{\alpha})$ (cf.\ also \cite[Thm.\ 4.3.8]{dilibertiLogicConcepts2category2025}). We thus proceed as in the coherent case: first we will characterise $\Syn^{\Hcal_\alpha}(\mc X)$ as the full subcategory of $\Xcal$ spanned by the {regular objects}, and then we will conclude that both $\Hcal_\alpha$ and $\Hmatte$ are conceptually complete by \Cref{prop:generic-cc}. For more properties of supercompact and regular objects and regular topoi implicitly used in the proofs below, we also refer the reader to \cite{rogersSUPERCOMPACTLYCOMPACTLYGENERATED2021}. 

\begin{lem}
    The following are equivalent:
    \begin{enumerate}
        \item every object of $\Syn^{\Hcal_\alpha}(\mc X)$ is supercompact in $\Xcal$;
        \item every object of $\Syn^{\Hcal_\alpha}(\mc X)$ is regular in $\Xcal$.
    \end{enumerate}
\end{lem}
\begin{proof}
    Analogous to \Cref{lem:syntactic-objects-are-compact-iff-coherent}.
\end{proof}

\begin{prop}\label{prop:supercompact-objects-in-regular-topos}
    An object $X \in \Xcal$ is supercompact if and only if, for each set $I$ and each geometric morphism $x \colon \Set^I \to \Xcal$, the canonical comparison 2-cell below $\sigma^X \colon X \circ \ran_{i_I} x \Rightarrow \ran_{i_I} (X \circ x)$ is an epimorphism (as maps in $\sh(\alpha I)$).
    \[\begin{tikzcd}[column sep = 35pt, row sep=25pt]
	{\Set^I} & {\mc X} \\
	{\sh(\alpha I)} & {\Set[\mbb O]}
	\arrow["x", from=1-1, to=1-2]
	\arrow["{i_I}"', hook, from=1-1, to=2-1]
	\arrow["X", from=1-2, to=2-2]
	\arrow[""{name=0, anchor=center, inner sep=0}, "{{\ran_{i_I} x}}"{description}, dashed, from=2-1, to=1-2]
	\arrow[""{name=1, anchor=center, inner sep=0}, "{{\ran_{I_i} (X\circ x)}}"', dashed, from=2-1, to=2-2]
	\arrow["\sigma^X"{pos=0.3}, between={0.2}{0.8}, Rightarrow, from=0, to=1]
    \end{tikzcd}\]
\begin{proof}
Analogous to \Cref{prop:compact-objects-in-coherent-topos}: instead of the ultrafilter $\mu$, here one should consider the point corresponding to $\infty$. In particular, we use the fact that if $X$ is supercompact then there exists a cover $e \colon Y \surj X$ where $Y$ is regular (which follows exactly as the analogous statement for compact and coherent objects in coherent topoi), and that regular topoi also have enough points as they are coherent.
\end{proof}
\end{prop}

\begin{thm}[Regular logic is conceptually complete]
    The doctrine $\ms T^{\mc H_\alpha}$ coincides with the free exact category construction over $\lex$, and $\mc H_{\alpha}$ is conceptually complete; the same holds for $\Hmatte$.
\begin{proof}
    Analogous to \Cref{flatcc}.
\end{proof}
\end{thm}

\subsection{Essentially algebraic logic with falsum}\label{ssec:cc-ess-alg-logic-w-falsum}
\looseness=-1
We now move on to the case of \emph{essentially algebraic logic with falsum} $\Hdom$ (\Cref{def:ess-alg-logic-with-falsum}). Following our general strategy, consider the \emph{free lex category with a strict initial object} doctrine $\T^{0}$; direct images of maps in $\Hdom$ are $\ms T^0$-algebra morphisms as they preserve the initial object.

As a convenient sublogic, consider the class $\mc H_{\mr{dense}} \subseteq \Hdom$ of \emph{dense embeddings}, i.e.\ dominant geometric morphisms which are also embeddings. We first describe a general family of dense embeddings: for any inhabited category $\mc C$, let $\mc C^{\rhd}$ be the category obtained by freely adjoining a terminal object to $\mc C$, which we denote by $\infty$.

\begin{lem}\label{lem:canonical-dense-maps}
     For each inhabited small category $\mc C$, the inclusion $\mc C \hook \mc C^{\rhd}$ induces a dense embedding $\eta \colon \psh(\mc C) \hook \psh(\mc C^{\rhd})$. 
\end{lem}
\begin{proof}
Concretely, we have to show that $\eta_* \colon \Psh(\mc C)\to \Psh(\mc C^{\rhd})$ preserves the initial object $0 \in \Psh(\mc C)$. First, note that the inverse image $\eta^* \colon \Psh(\mc C^{\rhd}) \to \Psh(\mc C)$ is obtained by restriction along the inclusion $\mc C\hook\mc C^{\rhd}$, which implies that for any $c\in\mc C$ we have $\eta^*\yo_c \cong \yo_c$, and hence:
    \[ \eta_*0 (c) \cong  \Psh(\mc C)(\eta^*\yo_c, 0) \cong \Psh(\mc C)(\yo_c, 0) \cong \emptyset \]
Then, since both $\yo$ and $\eta^*$ preserve finite limits, note that $\eta^* \yo_\infty$ is the terminal object $1 \in \Psh(\mc C)$. Thus, since $\Psh(\mc C)$ is non-degenerate as $\mc C$ is inhabited:
\[ \eta_* 0 (\infty) \cong  \Psh(\mc C)(\eta^*\yo_\infty, 0) \cong \Psh(\mc C)(1, 0) \cong \emptyset.\]
Therefore, $\eta_* 0 \cong 0$.
\end{proof}

\begin{rem}[Stalk at $\infty$ computes global sections]\label{rem:stalk-dense-global-sections}
    Note that, for any $X\in\psh(\mc C)$, the stalk of $\eta_*X$ at $\infty$ is simply given by the set of global sections of $X$:
    \[ (\eta_*X)_\infty \cong \eta_*X(\infty) \cong \psh(\mc C)(1,X). \]
\end{rem}

For a category $\mc C$ in $\lex$ having a strict initial object, its `classifying topos' is given by $\sh(\mc C, J_0)$ where $J_0$ is the topology having an empty covering on the initial object. Note that 
\[ \sh(\mc C, J_0) \simeq \psh(\mc C_+), \]
where $\mc C_+ \subseteq \mc C$ is the full subcategory of non-initial objects. In particular, the canonical embedding $\sh(\mc C, J_0) \simeq \psh(\mc C_+) \hook \psh(\mc C)$ is \emph{closed} (cf.\ \cite[Lem.\ 1.4.6]{dilibertiLogicConcepts2category2025}). Using the dense maps described in \Cref{lem:canonical-dense-maps}, we can calculate $\Hcal_{\mr{dense}}$-syntactic categories of topoi of this form.

\begin{prop}\label{prop:syn-cats-of-dense}
    For any category $\mc C$ in $\lex$ having a strict initial object, the canonical inclusion below is an equivalence:
    \[ \mc C \hook \Syn^{\mc H_{\mr{dense}}}(\sh(\mc C, J_0))\]
\end{prop}
\begin{proof}
First, note that the inclusion above is well-defined since by \Cref{lem:representables-lie-in-syntactic-category} we have an embedding $\mc C \hook \Syn^{\mc H_{\mr{dom}}}(\sh(\mc C, J_0))$, while $\Hcal_{\mr{dense}} \subseteq \Hdom$ gives an embedding $\Syn^{\mc H_{\mr{dom}}}(\sh(\mc C, J_0)) \hook \Syn^{\mc H_{\mr{dense}}}(\sh(\mc C, J_0))$. 

Fix $X\in \Sh(\mc C, J_0)$. If $X \cong 0$, then by the strictness of the initial object in $\mc C$ we clearly have that $X \cong \yo_0$. Assume then that $X$ is non-initial and suppose it lies in $\Syn^{\mc H_{\mr{dense}}}(\sh(\mc C, J_0))$. Through the equivalence $\Sh(\mc C, J_0) \simeq \Psh(\mc C_+)$, we can identify $X$ with a presheaf on $\mc C_+$. We will now show that $X$ is a retract of a representable in $\Psh(\mc C_+)$: since $\mc C$ is Cauchy-complete as it is left exact, and since the splitting of idempotents of a non-initial object is non-initial if the initial object is strict, we know that $\mc C_+$ is Cauchy-complete as well, so that in that case $X$ will itself be representable.

Consider the following diagram, where $e \colon \Psh(\mc C_+)/X \to \Psh(\mc C_+)$ is the canonical étale geometric morphism induced by slicing:
    \[\begin{tikzcd}
	& {\psh(\int_{\mc C_+}X) \simeq \psh(\mc C_+)/X } & {\psh(\mc C_+)} \\
	\Set & {\psh((\int_{\mc C_+}X)^\rhd)} & {\Set[\mbb O]}
	\arrow["e", from=1-2, to=1-3]
	\arrow["\eta"', hook, from=1-2, to=2-2]
	\arrow["X", from=1-3, to=2-3]
	\arrow["\infty"', from=2-1, to=2-2]
	\arrow[""{name=0, anchor=center, inner sep=0}, "{\ran_\eta e}"{description}, dashed, from=2-2, to=1-3]
	\arrow[""{name=1, anchor=center, inner sep=0}, "{\eta_*e^*X}"', dashed, from=2-2, to=2-3]
    \end{tikzcd}\]
Note that $e^* \colon \Psh(\mc C_+) \to \Psh(\mc C_+)/X$ acts by pulling back along the terminal map $X \to 1$. 
Therefore, writing $X$ as the (by assumption, inhabited) colimit $X \cong \colim_{(c, x) \in \int_{\mc C_+}X} \yo_c$, we have:
\[ (\ran_\eta e)^* X \cong \colim_{(c, x) \in \int_{\mc C_+}X} (\ran_\eta e)^*\yo_c \cong \colim_{(c, x) \in \int_{\mc C_+}X} \eta_* e^* \yo_c \]
Hence, by \Cref{rem:stalk-dense-global-sections}, the stalk of $(\ran_\eta e)^* X$ at $\infty$ can be computed as
\[ ((\ran_\eta e)^* X)_\infty \cong  \colim_{(c, x) \in \int_{\mc C_+}X} \Psh(\mc C_+)(X, \yo_c)\]
since, for each $c \in \mc C$, global sections of $e^*\yo_c$ --- i.e.\ the projection $X\times \yo_c \to X$ --- in $\Psh(\mc C_+)/X$ correspond simply to arrows $X \to \yo_c$ in $\Psh(\mc C_+)$.
On the other hand, again by \Cref{rem:stalk-dense-global-sections}, we have:
\[ (\eta_* e^*X)_\infty \cong \Psh(\mc C_+)(X,X)\]

Note then that, since $X$ is a morphism of $\Hcal_{\mr{dense}}$-injectives and $\eta$ is dense, $(\ran_\eta e)^* X \cong \eta_* e^* X$. In other words, this means that the canonical map 
\[ \colim_{(c, x) \in \int_{\mc C_+}X} \Psh(\mc C_+)(X, \yo_c) \to \Psh(\mc C_+)(X,X)\]
is a bijection: in particular, there exist some $c \in \mc C_+$ and some arrow $x \colon \yo_c \to X$ such that, for some $s \colon X \to \yo_c$, the composite $s \circ x$ is the identity. This means that $X$ is a retract of the representable $\yo_c$, which concludes the proof.
\end{proof}

\begin{rem}
Compare the previous to the proofs of~\cref{prop:compact-objects-in-coherent-topos} in the coherent case and~\cref{prop:supercompact-objects-in-regular-topos} in the regular case. A notable difference here is that, since in this case the fixed doctrine $\T^0$ does not endow a category with image factorisations, we are not able to reduce the argument to considerations on the lattice of subobjects of a sheaf $X \in \Sh(\mc C, J_0)$ and instead we have to consider its whole category of elements $\int_{\mc C}X$.

Recall that our goal is to show that, if $X$ lies in $\Syn^{\Hcal_{\mr{dense}}}(\Sh(\mc C, J_0))$, then $X$ actually belongs to $\mc C$. From the point of view of its category of elements, note that $X$ is representable if and only if $\int_{\mc C}X$ has a terminal object. Thus, we may freely adjoin a terminal object to $\int_{\mc C}X$ and observe the two right Kan extensions along the associated embedding $\Psh(\int_{\mc C} X) \hook \Psh((\int_{\mc C}X)^{\rhd})$: in this case, the embedding is dense when $X$ is non-initial, which is what allows us for a clean proof. As we will see in~\cref{ssec:cc-finit-disj-logic}, a similar argument will prove conceptual completeness for finitary disjunctive logic, but it will require a slightly more involved construction.
\end{rem}

\begin{rem}\label{rem:dense-maps-if-epi-then-iso}
The proof of~\cref{prop:syn-cats-of-dense} reveals that we actually proved a stronger claim: if the canonical comparison 2-cells $X \circ \ran_f g \Rightarrow \ran_f (X \circ g)$ for all dense embeddings $f$ are \emph{epimorphisms}, then in fact they are isomorphisms.
\end{rem}

\begin{thm}[Essentially algebraic logic with falsum is conceptually complete]
\looseness=-1
The doctrine $\ms T^{\mc H_{\mr{dense}}}$ coincides with the free lex category with a strict initial object construction over $\lex$, and $\Hcal_{\mr{dense}}$ is conceptually complete; the same holds for $\mc H_{\mr{dom}}$.
\begin{proof}
    Combine \Cref{prop:syn-cats-of-dense} with \Cref{prop:generic-cc}.
\end{proof}
\end{thm}

\subsection{Finitary disjunctive logic}\label{ssec:cc-finit-disj-logic}Finally, we consider the case of \emph{finitary disjunctive logic} $\Hpure$ (\Cref{def:finit-disj-logic}). This proof will proceed in formal analogy with the previous case, starting from the \emph{free lextensive category\footnote{Recall that a category is \emph{extensive} if it has finite coproducts which are moreover disjoint and pullback-stable, and it is \emph{lextensive} if it moreover admits finite limits.}} doctrine $\T^{\mr{lext}}$: since its algebras are lextensive categories, we have that $\Hpure$ coincides with the class of geometric morphisms whose direct images are $\T^{\mr{lext}}$-algebra morphisms. 

We note here that, for a lextensive category $\mc C$, its `classifying topos' is given by $\Sh(\mc C, J_{\mr{ext}})$ where $J_{\mr{ext}}$ is the \emph{extensive topology}, i.e.\ the topology generated by finite families of inclusions into a coproduct. Concretely,
\[ \sh(\mc C, J_{\mr{ext}}) \simeq \mathsf{FP}(\mc C\op,\Set), \]
where the right-hand side is the category of finite-product-preserving functors $\mc C\op \to \Set$.

In this case, too, as a convenient sublogic of $\Hpure$, we will consider the class $\Hpure'$ of pure embeddings. However, here it will take us some extra work to describe a general class of pure embeddings which we will use in order to characterise the $\Hpure$-syntactic category of a lextensive topos. We begin by recalling that, by~\cite[Lem.\ C3.4.12]{johnstoneSketchesElephantTopos2002}, a geometric morphism is pure if and only if its direct image preserves the internal Boolean algebra $2 \coloneqq 1 + 1$, and hence in the following lemma we describe the latter explicitly in toposes of the form $\Sh(\mc C, J_{\mr{ext}})$.

\begin{lem}\label{lem:dec-subs-of-extensive-object}
    For any lextensive category $\mc C$, the sheaf $2\in\sh(\mc C, J_{\mr {ext}})$ is computed as follows,
    \[ 2(c) \cong \mr{Sub}_{\mc C}^{\mr{dec}}(c), \]
    where $\mr{Sub}_{\mc C}^{\mr{dec}}(c)$ is the Boolean algebra of decidable subobjects of $c$ in $\mc C$.
\end{lem}
\begin{proof}
    Since the constant presheaf of value $2 = \{0,1\}$ is always separated,  note first that its sheafification is computed on each $c\in \mc C$ as the following colimit:
    \[ 2(c) \cong \colim_{c_1+\cdots+c_{n}\cong c}2^{n}. \]

    Consider then the function
    \[ f \colon 2(c) \cong \colim_{c_1+\cdots+c_{n}\cong c}2^{n} \to \mr{Sub}_{\mc C}^{\mr{dec}}(c), \]
    defined by mapping a decomposition $c_1+\cdots+c_{n}\cong c$ and an element of $2^{n}$, i.e.\ a subset $I \subseteq {n}$, to the coproduct $\coprod_{i\in I}c_i$ seen as a (decidable) subobject of $c$. On one hand, the map $f$ is evidently surjective, since any subobject $c_1 \inj c$ has a complement $c_2 \inj c$, so that $c_1$ is reached as the $f$-image of the decomposition $c_1 + c_2 \cong c$ together with the element $(1,0) \in 2^2$. On the other hand, to see that $f$ is also injective, suppose that $c_0+c_1$ and $c_0'+c_1'$ are two decomposition of a decidable subobject $d \inj c$; then, the decomposition 
    \[ \coprod_{i,j \in 2}c_i \cap c_j' \cong d\]
    refines both decompositions, which shows that the two decompositions are identified in the colimit, and hence $f$ agrees on them. The general case of two decompositions $c_1+\dots+c_n \cong d$ is analogous.
\end{proof}

Fix from here on an extensive category $\mc C$. We denote by $\mathrm{B}_{\mc C}$ the Boolean algebra of decidable subterminal objects in $\Sh(\mc C, J_{\mr{ext}})$, equivalently given as:
\[ \mathrm{B}_{\mc C} \coloneqq \sh(\mc C, J_{\mr{ext}})(1,2). \]
We define the \emph{join} category $\mc C \rhd \mathrm{B}_{\mc C}$ as follows:
\begin{itemize}
    \item[--] its objects are the disjoint union of the objects of $\mc C$ and the elements of $\mathrm{B}_{\mc C}$;
    \item[--] its morphisms are either the morphisms of $\mc C$, or a morphism in $\mathrm{B}_{\mc C}$, or a unique morphism $c \to U$ if there is a (necessarily unique) morphism $\yo_c \to U$ in $\Sh(\mc C, J_{\mr{ext}})$.
\end{itemize}

\begin{rem}
    Note that the join $\mc C \rhd \mathrm{B}_{\mc C}$ is \emph{not} the full subcategory of $\sh(\mc C, J_{\mr{ext}})$ spanned by the objects of $\mc C$ and the decidable subterminal objects. While in the former there are no morphisms going from elements of $\mathrm{B}_{\mc C}$ to objects of $\mc C$, in the latter there could be some, corresponding (for instance) to global sections of representables.
\end{rem}

We then define a topology $J_B$ on $\mc C \rhd \mathrm{B}_{\mc C}$ as follows:
\begin{itemize}
    \item[--] covering sieves of an object $c\in\mc C$ are given by $J_{\mr{ext}}$-covering sieves of $c$ in $\mc C$, i.e.\ sieves containing a finite family of inclusions $\set{ c_i \hook c }_{i\leq n}$ inducing an isomorphism $\coprod_{i\leq n}c_i\cong c$;
    \item[--] covering sieves of an element $U \in \mathrm{B}_{\mc C}$ are those sieves containing a finite pairwise-disjoint family $\set{ U_i }_{i\leq n}$ in $\mathrm{B}_{\mc C}$ such that $U = \bigvee_{i\leq n}U_i$.
\end{itemize}

\begin{lem}\label{lem:fin-disj-topology-well-def}
    The topology $J_B$ on $\mc C \rhd \mathrm{B}_{\mc C}$ is well-defined.
\begin{proof}
    Maximal sieves are clearly covering. To show stability under pullback, note first that covering sieves on $c \in \mc C$ may only be pulled back along a morphism of $\mc C$, and hence they are stable since $J_{\mr{ext}}$ is a topology on $\mc C$. Covering sieves on $U \in \mathrm{B}_{\mc C}$ are also trivially stable under pullback along a morphism of $\mathrm{B}_{\mc C}$, so that we can reduce to consider the pullback of a covering sieve $\set{ U_i }_{i\leq n}$ on $U$ along a morphism $c \to U$ for some $c \in \mc C$. By definition, this means that there is a map $\yo_c \to U$ in $\Sh(\mc C, J_{\mr{ext}})$: pulling back $\set{ U_i }_{i\leq n}$ along it in $\Sh(\mc C, J_{\mr{ext}})$, we obtain a family of decidable subobjects $\set{S_i \hook \yo_c}_{i\leq n}$ of $\yo_c$ in $\Sh(\mc C, J_{\mr {ext}})$ such that $\yo_c \cong \coprod_{i\leq n}S_i$. Since decidable subobjects of $\yo_c$ correspond precisely to morphisms $\yo_c \to 2$, i.e.\ to elements of $2(c)$, by~\cref{lem:dec-subs-of-extensive-object} we know that all such subobjects $S_i$ are of the form $\yo_{c_i}$ where $c_i \hook c$ is a decidable subobject in $\mc C$; hence, they are a $J_{\mr{ext}}$-cover of $c$ in $\mc C$. Similarly one can show that $J_B$ satisfies local character.
\end{proof}
\end{lem}

We can finally introduce the desired class of pure embeddings, which will be used to characterise the necessary $\Hpure$-syntactic categories towards an application of the reduction lemma.

\begin{lem}
    The canonical inclusion $\mc C \hook \mc C \rhd \mathrm{B}_{\mc C}$ induces a pure embedding 
    \[ \eta \colon \Sh(\mc C, J_{\mr{ext}}) \hook \Sh(\mc C \rhd \mathrm{B}_{\mc C}, J_B).\]
\begin{proof}
    \looseness=-1
    By construction, the inclusion $\iota_{\mc C} \colon \mc C \hook \mc C \rhd \mathrm{B}_{\mc C}$ is a comorphism of sites $(\mc C, J_{\mr{ext}}) \to (\mc C \rhd \mathrm{B}_{\mc C}, J_B)$, i.e. $\iota_{\mc C}$ is cover lifting, since if a sieve on $c$ in $\mc C$ is a covering in $\mc C \rhd \mathrm{B}_{\mc C}$, then it is also a covering in $\mc C$. Hence it induces a geometric morphism $\eta \colon \Sh(\mc C, J_{\mr{ext}}) \to \Sh(\mc C \rhd \mathrm{B}_{\mc C}, J_B)$. The fact that $\eta$ is an embedding follows by~\cite[Prop. 7.6]{caramelloDensenessConditionsMorphisms2020} from the fact that $\iota_{\mc C}$ is fully faithful. 

    To see that $\eta$ is pure, we can reduce to show that $\eta_* (2) \cong 2$. Note that for any $c \in \mc C$ we can identify $\eta^* \yo_c$ with the representable $\yo_c$, and hence:
    \[ \eta_*2 (c)  \cong \Sh(\mc C, J_{\mr{ext}})(\yo_c, 2) \cong \Sub_{\mc C}^{\mr{dec}}(c) \cong \Sh(\mc C \rhd \mathrm{B}_{\mc C} , J_B )(\yo_c , 2) \cong 2(c), \]
    where the third isomorphism follows by \Cref{lem:dec-subs-of-extensive-object} since $J_B$-covers of $c$ in $\mc C \rhd \mathrm{B}_{\mc C}$ correspond exactly to $J_{\mr{ext}}$-covers of $c$ in $\mc C$. On the other hand, for any $U \in \mathrm{B}_{\mc C}$, by construction we have 
    \[ \eta^* \yo_U (c) \cong (\mc C \rhd \mathrm{B}_{\mc C})(c, U) \cong \Sh(\mc C, J_{\mr{ext}})(\yo_c, U) \cong U(c)\]
    for each $c\in \mc C$, i.e.\ $\eta^*\yo_U \cong U$, from which:
    \[ \eta_*2 (U) \cong \Sh(\mc C, J_{\mr {ext}})( U , 2) \cong \mathrm{B}_{\mc C}/U \cong \Sh(\mc C \rhd \mathrm{B}_{\mc C})(U, 2)\cong 2(U),\]
    where the second isomorphism follows by construction while the third isomorphism follows again by \Cref{lem:dec-subs-of-extensive-object}. It follows that $\eta_*2 \cong 2$, and hence $\eta$ is pure.
\end{proof}
\end{lem}

\begin{prop}\label{prop:syn-cats-of-pure-embeddings}
    For any lextensive category $\mc C$, the canonical inclusion below is an equivalence:
    \[ \mc C \hook \Syn^{\mc H'_{\mr{pure}}}(\sh(\mc C, J_{\mr{ext}})). \]
\end{prop}
\begin{proof}
    As for \Cref{prop:syn-cats-of-dense}, note first that the inclusion above is well-defined by \Cref{lem:representables-lie-in-syntactic-category}. Fix then $X \in \Syn^{\Hpure'}(\Sh(\mc C, J_{\mr {ext}}))$: as in the proof of \emph{ibid.}, we will now show that $X$ is a retract of a representable in $\Sh(\mc C, J_{\mr {ext}})$: since $\mc C$ is Cauchy-complete as it is left exact, $X$ itself will be representable.

    Consider the category of elements $\int_{\mc C}X$ of $X$: the fibration $\int_{\mc C}X \to \mc C$ preserves and creates finite coproducts (since $\int_{\mc C}X$ can be viewed as the restricted slice category $\mc C/X$), and hence $\int_{\mc C}X$ is extensive. The sheaf topos on $\int_{\mc C}X$ with the extensive topology is, by construction, the \'etale topos over $X$:
        \[\textstyle \sh(\int_{\mc C}X, J_{\mr{ext}}) \simeq \sh(\mc C, J_{\mr{ext}})/X. \]
    Let $B$ the Boolean algebra of decidable subterminal objects in $\sh(\int_{\mc C}X, J_{\mr{ext}})$ and consider the following diagram, where $e \colon \sh(\mc C, J_{\mr{ext}})/X \to \sh(\mc C, J_{\mr{ext}})$ is the canonical \'etale geometric morphism induced by slicing:
\[\begin{tikzcd}
	{\sh(\int_{\mc C}X, J_{\mr{ext}}) \simeq \Sh(\mc C, J_{\mr{ext}})/X} & {\Sh(\mc C, J_{\mr{ext}})} \\
	{\sh(\int_{\mc C}X \rhd  B, J_B)} & {\Set[\mbb O]}
	\arrow["e", from=1-1, to=1-2]
	\arrow["\eta"', hook, from=1-1, to=2-1]
	\arrow["X", from=1-2, to=2-2]
	\arrow[""{name=0, anchor=center, inner sep=0}, "{{\ran_\eta e}}"{description}, dashed, from=2-1, to=1-2]
	\arrow[""{name=1, anchor=center, inner sep=0}, "{{\eta_*e^*X}}"', dashed, from=2-1, to=2-2]
\end{tikzcd}\]

Writing $X$ as the colimit $X \cong \colim_{(c, x) \in \int_{\mc C}X} \yo_c$, we have
\[ (\ran_\eta e)^* X  \cong \colim_{(c, x) \in \int_{\mc C}X} \eta_* e^* \yo_c. \]
To evaluate this colimit at $1 \in B$, we may compute its value as the sheafification of the colimit in the presheaf category as follows,
    \begin{align*}
        ((\ran_\eta e)^* X)(1)
        &\cong \colim_{U_1+\cdots+U_n\cong X}\prod_{i\le n} \colim_{(c, x) \in \int_{\mc C}X} \sh(\mc C, J_{\mr{ext}})/X(U_i,e^*\yo_c) \\
        &\cong \colim_{(c, x) \in \int_{\mc C}X} \colim_{U_1+\cdots+U_n\cong X}\prod_{i\le n}\sh(\mc C, J_{\mr {ext}})(U_i,\yo_c)
    \end{align*}
Here the first isomorphism holds since, by construction, the coverings of $1$ in $\int_{\mc C}X \rhd B$ are decompositions of $X$ via decidable subobjects $U_1,\dots, U_n$; the second isomorphism holds since $\int_{\mc C}X$ is extensive, thus in particular sifted, and hence colimits over it commute with finite products (and clearly with other colimits). 

On the other hand, the value of $\ran_{\eta} (X\circ e) \cong \eta_* e^* X$ at $1$ is computed as:
\[\textstyle \eta_*e^* X (1) \cong \sh(\int_{\mc C}X, J_{\mr{ext}})(1, e^*X) \cong \sh(\mc C, J_{\mr{ext}})(X,X). \]
Thus, since $X$ is a morphism of $\Hpure'$-injectives and $\eta$ is a pure embedding, $(\ran_\eta e)^* X \cong \eta_* e^* X$. 
In other words, this means that the canonical map
\[ \colim_{(c, x) \in \int_{\mc C}X} \colim_{U_1+\cdots+U_n\cong X}\prod_{i\le n}\sh(\mc C, J_{\mr {ext}})(U_i,\yo_c) \to \sh(\mc C, J_{\mr{ext}})(X,X) \]
is a bijection: in particular, there exist some $c \in \mc C$ and some decomposition $U_1 + \cdots + U_n \cong X$ with maps $u_i \colon U_i \to \yo_c$ such that, for some $s \colon \yo_c \to X$, the composite
\[
\begin{tikzcd}[column sep = 50pt]
    X \cong U_1 + \cdots + U_n \ar[r, "{[u_1,\cdots,u_n]}"] & \yo_c \ar[r, "s"] & X
\end{tikzcd}
\]
is the identity. This means that $X$ is a retract of the representable $\yo_c$, which concludes the proof.
\end{proof}

\begin{rem}
    As in \Cref{rem:dense-maps-if-epi-then-iso}, note that here, too, we actually proved a stronger claim. Indeed, it is easy to see that the terminal object in $\sh(\mc C \rhd \mathrm{B}_{\mc C}, J_B)$ is always projective, meaning that its global section functor preserves epimorphisms; therefore, the previous proof shows if the canonical 2-cells $X \circ \ran_f g \Rightarrow \ran_f (X \circ g)$ for all pure embeddings $f$ are \emph{epimorphisms}, then in fact they are isomorphisms.
\end{rem}

\begin{thm}[Finitary disjunctive logic is conceptually complete]
    The doctrine $\ms T^{\Hpure'}$ coincides with the free lextensive category construction over $\lex$, and $\Hpure'$ is conceptually complete; the same holds for $\Hpure$.
\begin{proof}
    Combine \Cref{prop:syn-cats-of-pure-embeddings} with \Cref{prop:generic-cc}.
\end{proof}
\end{thm}

\subsection{Other examples and remarks} \label{sec:comments}

In this subsection we make some reflections on the previous results and we discuss possible further examples. The examples and the reflections often come with different questions in mind concerning where to bring our theory next.

One immediate observation is that, currently, all our examples of conceptually complete logics sit `below' the coherent fragment, in the sense that their monad of syntax is a submonad of $\T^{\Hflat}$. However, we do not currently believe that this limitation is necessary towards a proof of conceptual completeness. For instance, the example below introduces very natural generalisations of the coherent fragment: while these fragments sit above the coherent one, they are still bounded, and thus amenable to our theory.
    
\begin{exa}[$\kappa$-coherent logic]For a fixed regular cardinal $\kappa$, we denote by $\Hcal_{\kappa\text{-flat}}$ the class of geometric morphisms whose direct images preserve $\kappa$-small coproducts and regular epimorphism. This is a $\kappa$-parametric generalisation of coherent logic for which we could `guess' its monad of syntax $\mathsf{T}^{\Hcal_{\kappa\text{-flat}}}$ to be the free $\kappa$-pretopos construction on a lex category, and try to prove conceptual completeness by applying the reduction lemma.
\end{exa}

It would also be interesting to explore some \emph{unbounded} logics $\Hcal$, and see to what extent one can study conceptual completeness for them in the sense that the classifying topos construction is fully faithful as a 2-functor on \emph{large} $\T^{\Hcal}$-algebras (cf.\ \cite[Rem.\ 6.3.8]{dilibertiLogicConcepts2category2025}). Besides full geometric logic, which is then trivially conceptually complete (cf.\ \cite[Rem.\ 6.2.2]{dilibertiLogicConcepts2category2025}), consider the examples below.

\begin{exa}[$\kappa$-disjunctive and infinitary disjunctive logic]
\looseness=-1
For a fixed regular cardinal $\kappa$, we denote by $\Hcal_{\kappa\text{-inn}}$ the class of geometric morphisms whose direct images preserve $\kappa$-small coproducts; more generally, we denote by $\Hcal_{\mr{inn}}$ the unbounded case, where direct images preserve all small coproducts.
These are infinitary generalisations of $\Hpure$, the latter of which was introduced in \cite{dilibertiLogicConcepts2category2025} under the name of \emph{infinitary disjunctive logic}. Currently, we believe that $\Hcal_{\text{inn}}$ should behave quite similarly to its finitary fragment, and should enjoy some form of conceptual completeness.
\end{exa}

Another very interesting direction would be that of \textit{exotic} fragments of geometric logic, as in, examples which are not inspired by the tradition of categorical logic. Is there any chance we can have genuinely new examples? And if so, what is their logical interpretation? In \cite[\S 7.3]{dilibertiLogicConcepts2category2025}, the authors tie this question with a \textit{call to syntax}, a quest for a calculus that is yet to be developed. For the moment, we present a new interesting example.

\begin{exa}
    For a fixed regular cardinal $\kappa$, we denote by $\Hcal_{\kappa}$ the class of geometric morphisms whose direct image preserves all $\kappa$-small colimits.\footnote{This logic was initially considered in the first draft of \cite{dilibertiGeometryCoherentTopoi2022}, essentially by mistake.} For $\kappa > \omega$, this logic coincides with $\Hcal_{\kappa\text{-flat}}$: this is due to the fact that if a category admits countable coproducts and effective quotients, then it also admits all coequalisers, and these coequalisers are preserved by $\kappa$-coherent functors\footnote{Indeed, for an arbitrary parallel pair of morphisms one can construct the smallest equivalence relation containing this pair via a countable union, and the coequaliser for the original pair coincides with the coequaliser for this equivalence relation.}. When $\kappa = \omega$, instead, we obtain a genuinely new logic, for which very little is understood. At the moment, we do not have a logical interpretation for this logic, nor we know if we should expect it to be conceptually complete, nor we can even guess its monad of syntax.
\end{exa}

Staying on the stream of exotic examples, a logic for which we currently have no logical understanding, but for which we foresee some logical meaning is \textit{integral logic}, which we shall recall below.

\begin{exa}[Integral logic]
\looseness=-1
    We denote by $\Hcal_{\text{cl}}$ the logic defined by the single geometric morphism $0 \to \Set$, introduced in \cite{dilibertiLogicConcepts2category2025} under the name of \emph{integral logic}. 
\end{exa}

\begin{rem}[Is there any non-conceptually-complete logic?]
\looseness=-1
At this point, we naturally wonder whether there is any fragment of geometric logic that is not conceptually complete. Although we currently do not have an answer, this question will strongly inspire the next section.
\end{rem}

\section{Conservatively embedded logics}\label{sec:cons-emb}
\label{ssec:conservatively-embedded}
\looseness=-1{
In \Cref{sec:everyalgisasyncat} we provided a practical criterion to prove conceptual completeness for a fragment of geometric logic, and in \Cref{sec:examples} we employed this criterion to prove that a number of fragments are indeed conceptually complete. In this section, we focus on a further exploration of the general theme of conceptual completeness, and we derive a new proof-theoretic property of conceptually complete logics.

We start by taking a step back and considering the adjunction established in \Cref{sec:everyalgisasyncat} for a bounded logic $\Hcal$, which is a reflection if and only if the logic is conceptually complete. 
\[\begin{tikzcd}
	{\WRInj(\Hcal)} && {{\alg(\T^{\Hcal})}\op}
	\arrow[""{name=0, anchor=center, inner sep=0}, "{\Syn^{\Hcal}}", curve={height=-12pt}, from=1-1, to=1-3]
	\arrow[""{name=1, anchor=center, inner sep=0}, "\Cl^{\Hcal}", curve={height=-12pt}, from=1-3, to=1-1]
	\arrow["\dashv"{anchor=center, rotate=-90}, draw=none, from=0, to=1]
\end{tikzcd}\]
\looseness=-1 A priori, a less ambitious task to investigate is: \textit{when is this adjunction idempotent?} Of course, this is true under the assumption of conceptual completeness, since reflections are indeed idempotent. It turns out that logics $\Hcal$ such that the adjunction $\Syn^{\Hcal}\dashv \Cl^{\Hcal}$ is idempotent are characterised by a distinctive proof-theoretic behaviour, which we shall refer to as being \textit{conservatively embedded} in geometric logic.

We single out the class of (bounded) logics $\Hcal$ such that the adjunction $\Syn^{\Hcal}\dashv \Cl^{\Hcal}$ is idempotent, which we will define as being \textit{conservatively embedded}. To justify our choice of terminology, we introduce the \emph{$\Hcal$-theory} of a topos formally in $\Hcal$: this will allow us to give a proof-theoretic interpretation to the definition of conservatively embedded logic.

There are many equivalent ways to characterise idempotent adjunctions. Here we focus on establishing when the component of the unit of the adjunction $\Syn^{\Hcal}\dashv \Cl^{\Hcal}$, at the classifying topos $\Cl^{\Hcal}(\mc A)$ of a small $\T^{\Hcal}$-algebra, is an equivalence of topoi:
\[p_{\Cl^{\Hcal}(\mc A)} \colon \Cl^{\Hcal}(\mc A) \to \Cl^{\Hcal}\Syn^{\Hcal}\Cl^{\Hcal}(\mc A).\] 
We shall start from a preliminary observation.

\begin{lem}\label{lem:unit-is-embedding}
    Let $\Hcal$ be a bounded logic. For each small $\T^{\Hcal}$-algebra $\mc A$, the geometric morphism $p_{\Cl^{\Hcal}(\mc A)} \colon \Cl^{\Hcal}(\mc A) \to \Cl^{\Hcal}\Syn^{\Hcal}\Cl^{\Hcal}(\mc A)$ is an embedding.
\end{lem}
\begin{proof}
    For ease of notation, let $\Xcal \coloneqq \Cl^{\Hcal}(\mc A)$. Consider the naturality square of the unit $p \colon 1 \Rightarrow \Cl^{\Hcal} \circ \Syn^{\Hcal}$ with respect to the embedding $j\colon \Xcal \inj \Psh(\mc A)$:
\[\begin{tikzcd}
    {\Xcal} & {\Cl^{\Hcal}  \Syn^{\Hcal}(\Xcal)} \\
	{\Psh (\mc A)} & {\Cl^{\Hcal} \Syn^{\Hcal}\Psh (\mc A)}
	\arrow["{p_{\Xcal}}", from=1-1, to=1-2]
	\arrow["{j}"', hook, from=1-1, to=2-1]
	\arrow["{\Cl^{\Hcal}\Syn^{\Hcal}(j)}", from=1-2, to=2-2]
	\arrow["{p_{\Psh(\mc A)}}"', from=2-1, to=2-2]
\end{tikzcd}\]

Note first that the morphism $p_{\Psh(\mc A)} \colon \Psh (\mc A) \to \Cl^{\Hcal} \Syn^{\Hcal}\Psh (\mc A)$ is an equivalence. Indeed, recall that by definition $\T^{\Hcal}(\mc A) = \Syn^{\Hcal}(\Psh(\mc A))$; then, 
by Diaconescu's equivalence (\Cref{thm:diaconescu}) we have that, for each topos $\mc E$:
\[ \Topoi(\mc E,\Cl^{\Hcal}(\ms T^{\mc H}\mc A)) \simeq \Alg(\ms T^{\mc H})(\ms T^{\mc H}\mc A,\mc E) \simeq \LEX(\mc A,\mc E) \simeq \Topoi(\mc E,\psh(\mc A)), \]
which implies that $\Cl^{\Hcal}(\ms T^{\mc H}(\mc A)) \simeq \psh(\mc A)$.\footnote{In particular, the equivalence is realised by $p_{\Psh(\mc A)}$ specifically, since the adjunction $\Syn^{\Hcal}\dashv \Cl^{\Hcal}$ is derived by Diaconescu's equivalence.} Therefore, since  $j$ is an embedding, the diagonal of the above square is an embedding too, and hence so is $p_{\Xcal}$.
\end{proof}
\begin{rem}
    More generally, the unit $p_{\Xcal} \colon \Xcal \to \Cl^{\Hcal}\Syn^{\Hcal}(\Xcal)$ is an embedding for every topos $\Xcal$ admitting an $\Hcal$-presentation (see \Cref{rem:classifying-topoi-have-presentations}).
\end{rem}

\begin{defn}[Conservatively embedded logics]\label{def:cons-emb-logic}
    A bounded logic $\Hcal$ is \emph{conservatively embedded} if, for every small $\T^{\Hcal}$-algebra $\mc A$, the unit $p_{\Cl^{\Hcal}(\mc A)} \colon \Cl^{\Hcal}(\mc A) \to \Cl^{\Hcal}\Syn^{\Hcal}\Cl^{\Hcal}(\mc A)$ of $\Syn^{\Hcal}\dashv \Cl^{\Hcal}$ at $\Cl^{\Hcal}(\mc A)$ is a geometric surjection, and hence an equivalence by \Cref{lem:unit-is-embedding}.
\end{defn}

\begin{constr}[A site interpretation of conservative embedding]
In the next proposition we shall give a characterisation of conservatively embedded logics. In order to familiarise with the notion, consider the geometric morphism $a_{\Cl^{\Hcal}(\mc A)}$ of \Cref{prop:syntactic-categories-are-algebras} and note that we can factor it in two ways, as depicted in the following diagram.
    \[\begin{tikzcd}[row sep = small, column sep =35pt]
	& {\Sh(\Syn^{\mathcal H}(\Cl^{\Hcal} \mc A), D^{\mathcal H}_{\Cl^{\Hcal}(\mc A)})} & \\
	{\Cl^{\Hcal}(\mc A)} && {\Psh(\Syn^{\Hcal}(\Cl^{\Hcal} \mc A ))} \\
	& {\Cl^{\Hcal}(\Syn^{\Hcal}(\Cl^{\Hcal} \mc A ))}
	\arrow[bend left=15, hook, from=1-2, to=2-3]
	\arrow["{{{q_{\Cl^{\Hcal}(\mc A)}}}}", bend left=15, two heads, from=2-1, to=1-2]
	\arrow["{{a_{\Cl^{\Hcal}(\mc A)}}}"{description}, dashed, from=2-1, to=2-3]
	\arrow["{{p_{\Cl^{\Hcal}(\mc A)}}}"', bend right=13, from=2-1, to=3-2]
	\arrow[bend right=11, hook, from=3-2, to=2-3]
\end{tikzcd}\]

By the universal property of the surjection-embedding factorisation, we obtain a geometric morphism $\Sh(\Syn^{\Hcal}(\Cl^{\Hcal}\mc A), D^{\Hcal}_{\Cl^{\Hcal}(\mc A)}) \to \Cl^{\Hcal}(\Syn^{\Hcal}(\Cl^{\Hcal}\mc A))$. By direct inspection, we see that this morphism is induced by a morphism of sites based on the identity functor,
\[t_{\mc A}\colon (\Syn^{\Hcal}(\Cl^{\Hcal} \mc A), J^{\Hcal}_{\Syn^{\Hcal}(\Cl^{\Hcal} \mc A)}) \to  (\Syn^{\Hcal}(\Cl^{\Hcal} \mc A),  D^{\Hcal}_{\Cl^{\Hcal}(\mc A)}).\]
The fact that $t_{\mc A}$ is a morphism of sites corresponds to the fact that the two defining topologies satisfy $J^{\Hcal}_{\Syn^{\Hcal}(\Cl^{\Hcal} \mc A)} \subseteq D^{\Hcal}_{\Cl^{\Hcal}(\mc A)}$. 
\end{constr}

\begin{prop}[Characterisations of conservative emdeddedness]\label{prop:conservatively-embedded-equivalent-defs}
    For a bounded logic $\Hcal$, the following are equivalent:
    \begin{enumerate}
        \item $\Hcal$ is conservatively embedded;
        \item the adjunction $\Syn^{\Hcal} \dashv \Cl^{\Hcal}$ is idempotent;
        \item $J^{\Hcal}_{\Syn^{\Hcal}(\Cl^{\Hcal} \mc A)} = D^{\Hcal}_{\Cl^{\Hcal}(\mc A)}$ for each small $\T^{\Hcal}$-algebra $\mc A$.
    \end{enumerate}
\end{prop}
\begin{proof}
    $(1) \Leftrightarrow (2)$ is clear. To see why $(1)\Leftrightarrow (3)$ holds, we go back to the diagram discussed in the previous construction for each small $\T^{\Hcal}$-algebra $\mc A$.
 Indeed, $p_{\Cl^{\Hcal}(\mc A)}$ is a surjection if and only if the canonical geometric morphism 
\[ \Sh(\Syn^{\Hcal}(\Cl^{\Hcal}\mc A), D^{\Hcal}_{\Cl^{\Hcal}(\mc A)} ) \to \Cl^{\Hcal}(\Syn^{\Hcal}(\Cl^{\Hcal} \mc A))\]
is an equivalence. Since this morphism is induced by the inclusion of topologies $J^{\Hcal}_{\Syn^{\Hcal}(\Cl^{\Hcal} \mc A)} \subseteq D^{\Hcal}_{\Cl^{\Hcal}(\mc A)}$, it is an equivalence if and only if the two topologies coincide.
\end{proof}

In order to unveil the logical meaning of \Cref{def:cons-emb-logic} and explain our choice of terminology, we now introduce some explicit \emph{syntax}: for the following definitions, fix a bounded logic $\Hcal$ and a topos $\Xcal$ formally in $\Hcal$. First, we introduce a geometric theory canonically associated to $\Xcal$ which encodes the information of the topology $D^{\Hcal}_{\Xcal}$ on $\Syn^{\Hcal}(\mc X)$.

\begin{defn}[$\mathcal{H}$-theory of a topos]\label{def:Htheory}
    We define the \emph{$\mc H$-theory} of $\mc X$ as the geometric theory $\mathbb{T}^{\mathcal{H}}_{\mc X}$ of signature $\Sigma^{\Hcal}_{\Xcal}$ constructed as follows.
    \begin{enumerate}
        \item The signature $\Sigma^{\Hcal}_{\Xcal}$ consists of one sort $X$ for each object $X$ in $\Syn^{\mc H}(\mc X)$, one function symbol $f$ from sort $X$ to sort $Y$ for each arrow $f\colon X \to Y$ in $\Syn^{\mc H}(\mc X)$, and one relation symbol $R_\phi$ of sort $X$ for each subobject $\phi \in \Sub_{\Xcal}(X)$.

        \item\looseness=-1 The theory $\mathbb{T}^{\mathcal{H}}_{\mc X}$ extends the canonical cartesian theory associated to the category $\Syn^{\mc H}(\mc X)$ as in~\cite[Ex.\ D1.4.8]{johnstoneSketchesElephantTopos2002}, i.e.\ the theory axiomatizing the finite-limits structure of $\Syn^{\mc H}(\mc X)$, with the following axioms:
        \begin{enumerate}[label={--}]
            \item $R_{\Delta_X}(x,x') \dashv\vdash_{x,x':X} x = x'$ for each diagonal $\Delta_X \colon X \inj X \times X$;
            \item $R_\phi(x) \vdash_{x:X} R_\psi(x)$ for each $\phi,\psi \in \Sub_{\mc X}(X)$ such that $\phi\leq \psi$;
            \item $\bigwedge_{i\le n}R_{\phi_i}(x) \vdash_{x:X} R_{\bigwedge_{i\le n}\phi_i}(x)$ for each finite family $\stt{\phi_i}_{i\le n}$ in $\Sub_{\mc X}(X)$; 
            \item $R_{\bigvee_{i\in I}\phi_i}(x) \vdash_{x:X} \bigvee_{i\in I}R_{\phi_i}(x)$ for each family $\stt{\phi_i}_{i\in I}$ in $\Sub_{\mc X}(X)$.
            \item $R_\phi(f(x)) \dashv\vdash_{x:X} R_{f^*\phi}(x)$ for each $f \colon X \to Y$ and each $\phi\in\Sub_{\mc X}(Y)$;
            \item $\exists x\!:\!X.fx = y \wedge R_{\phi}(y) \dashv\vdash_{y:Y} R_{\exists_f\phi}(y)$ for each $f \colon X \to Y$ and each $\phi\in\Sub_{\mc X}(Y)$.
        \end{enumerate}
    \end{enumerate}
\end{defn}

\begin{rem}[$\mathbb T^{\Hcal}_{\Xcal}$ encodes $D^{\Hcal}_{\Xcal}$]
    In particular, under this definition, a family of maps $\set{f_i\colon X_i \to X}_{i\in I}$ in $\Syn^{\mc H}(\mc X)$ is a covering family in $\mc X$ if and only if the following sequent is provable in $\mbb T^{\mc H}_{\mc X}$:
\[ \top \vdash_{x:X} \bigvee_{i\in I}\exists x_i : X_i \, .\, f_i(x_i) = x. \]
\end{rem}

\begin{rem}[The geometric syntactic category of $\mbb T^{\mc H}_{\mc X}$]
    It follows from~\cref{def:Htheory} that the syntactic category of $\mbb T^{\mc H}_{\mc X}$ as a geometric theory (cf.~\cite[\S 1.4]{caramelloTheoriesSitesToposes2018}) is equivalent to the full subcategory of $\mc X$ whose objects are subobjects of some $X\in\Syn^{\mc H}(\mc X)$. If $\Syn^{\mc H}(\mc X)$ is \emph{generating} in $\mc X$, then each relation in $\mc X$ on $X\in\Syn^{\mc H}(\mc X)$ can be written as a union of images of maps in $\Syn^{\mc H}(\mc X)$, thus the relation symbols added in~\cref{def:Htheory} will be redundant. In this case the classifying topos of $\mbb T^{\mc H}_{\mc X}$ will be $\mc X$ itself.
\end{rem}

Completely similarly, we also introduce a theory which encodes the information of the topology $J^{\Hcal}_{\Syn^{\Hcal}(\Xcal)}$ on $\Syn^{\Hcal}(\Xcal)$.

\begin{defn}[$\mc H$-theory of a syntactic category]
We define the \emph{$\Hcal$-theory} of $\Syn^{\Hcal}(\mc X)$ as the geometric theory $\mathbb T ^{\Hcal}_{\Syn^{\Hcal}(\mc X)}$ constructed as follows.
\begin{enumerate}
    \item The signature $\Sigma^{\Hcal}_{\Syn^{\Hcal}(\mc X)}$ consists of one sort $X$ for each object $X$ in $\Syn^{\Hcal}(\mc X)$ and one function symbol $f$ from $X$ to sort $Y$ for each arrow $f \colon X \to Y$ in $\Syn^{\Hcal}(\Xcal)$.
    \item\looseness=-1 The theory $\mathbb{T}^{\mathcal{H}}_{\Syn^{\Hcal}(\mc X)}$ extends the canonical cartesian theory associated to the category $\Syn^{\mc H}(\mc X)$ with the axiom
    \[ \top \vdash_{x:X} \bigvee_{i\in I}\exists x_i : X_i \, .\, f_i(x_i) = x.\]
    for each $J^{\Hcal}_{\Syn^{\Hcal}(\mc X)}$-covering $\set{f_i \colon X_i \to X}_{i\in I}$ of each $X$ in $\Syn^{\Hcal}(\Xcal)$.
\end{enumerate}
\end{defn}

\begin{rem}[$\mbb T^{\mc H}_{\mc X}$ vs.\ $\mbb T^{\mc H}_{\Syn^{\mc H}(\mc X)}$]
In essence, the difference between $\mbb T^{\mc H}_{\mc X}$ and $\mbb T^{\mc H}_{\Syn^{\mc H}(\mc X)}$ --- besides the additional axioms --- is that the former contains many more relation symbols. This is due to the fact that $\mbb T^{\mc H}_{\mc X}$ is a restriction of the canonical geometric theory of the topos $\mc X$ to the sorts lying in $\Syn^{\Hcal}(\mc X)$, thus we are canonically including all the subobjects that can appear there.
\end{rem}

Finally, note that by construction there is an evident interpretation $I_{\mc X} \colon \mbb T^{\mc H}_{\Syn^{\mc H}(\mc X)} \to \mbb T^{\mc H}_{\mc X}$, realised by the identity map on sorts and function symbols. The fact that the two topologies coincide is therefore equivalent to this interpretation being \emph{conservative}, i.e.\ such that geometric sequents in $\Sigma^{\Hcal}_{\Syn^{\Hcal}(\Xcal)} \subseteq \Sigma^{\Hcal}_{\Xcal}$ that are provable in $\mathbb T^{\Hcal}_{\Xcal}$ are already provable in $\mbb T^{\mc H}_{\Syn^{\mc H}(\mc X)}$. This gives the following proof-theoretic interpretation of conservative embeddedness, and justifies its name.

\begin{prop}[Conservatively embedded logics, syntactically]\label{prop:ce-logics-via-syntax}
   For a bounded logic $\Hcal$, the following are equivalent:
   \begin{enumerate}
   \item $\Hcal$ is conservatively embedded;
   \item for each small $\T^{\Hcal}$-algebra $\mc A$, the interpretation $I_{\Cl^{\Hcal}(\mc A)} \colon \mbb T^{\mc H}_{\Syn^{\mc H}(\Cl^{\Hcal} \mc A)} \to \mbb T^{\mc H}_{\Cl^{\Hcal} \mc A}$ is conservative.
   \end{enumerate}
    \begin{proof}
        Combine \Cref{prop:conservatively-embedded-equivalent-defs} with \cite[Cor.\ 6.2]{caramelloDensenessConditionsMorphisms2020}.
    \end{proof}
\end{prop}

\section{It was Makkai all along}\label{sec:makkai-all-along}

We conclude our analysis by discussing how our notion of a conceptually complete logic $\Hcal$ relates to Makkai's original conceptual completeness theorem for coherent logic \cite{makkaiStoneDualityFirst1987}, which we recall here. 

\begin{thm}[Makkai's conceptual completeness]\label{thm:makkai}
For each small pretopos $\mc C$, the category $\mathsf{Mod}(\mc C)$ of pretopos functors $\mc C \to \Set$ is an \emph{ultracategory}, and the category of \emph{ultrafunctors} $\mathsf{Mod}(\mc C) \to \Set$ is equivalent to $\mc C$ itself.
\end{thm}

The aim of this section is twofold. On one hand, we will show that conceptual completeness for the logic $\mc H_\beta$ in the sense of \Cref{def:cc-logic} implies, and is in fact equivalent to, Makkai's result (\Cref{cor:recovering-makkai}). On the other, we will give a general definition of \emph{conceptual completeness à la Makkai} (\Cref{def:cc-à-la-makkai}) which is more faithful to its traditional \textit{semantic} taste, and we will prove that it is equivalent to \Cref{def:cc-logic} under the assumption of completeness with respect to set-based models. Towards this aim, we shall start by analysing the three independent works~\cite{saadiaExtendingConceptualCompleteness2025, hamadGeneralisedUltracategoriesConceptual2025, vangoolToposesEnoughPoints2026}; before recalling the relevant definitions in~\cref{virtualultracat}, let us explain our perspective in this section.

The main result of \emph{ibid.} is that there is an idempotent 2-adjunction between the 2-category of toposes and that of \emph{(bounded) virtual ultracategories}\footnote{Strictly speaking, the notion appearing in \cite{hamadGeneralisedUltracategoriesConceptual2025} is slightly more general than the virtual ultracategories of \cite{saadiaExtendingConceptualCompleteness2025}, appearing in \cite{vangoolToposesEnoughPoints2026} under the name of \emph{ultraconvergence spaces}, cf.\ the discussion in \cite[\S 6.1]{aristoteProfunctorialAlgebras2026}.} and \emph{virtual ultrafunctors}, realised by homming into $\Set$, which restricts to a \emph{reflection} on \emph{toposes with enough points}. In analogy with the topological case, we dub \emph{sober} those virtual ultracategories in the essential image of the right adjoint.
\[\begin{tikzcd}
	\Topoi && {\mathsf{vUlt^{bnd}}} \\
	\\
	{\Topoi_{\mathsf{wep}}} && {\mathsf{vUlt^{bnd}_{sob}}}
	\arrow[""{name=0, anchor=center, inner sep=0}, "{{{{\mathsf{pt}}}}}"', curve={height=12pt}, from=1-1, to=1-3]
	\arrow[""{name=1, anchor=center, inner sep=0}, "{{{{\Ocal}}}}"', curve={height=12pt}, from=1-3, to=1-1]
	\arrow[hook, from=3-1, to=1-1]
	\arrow[""{name=2, anchor=center, inner sep=0}, "{{{{\mathsf{pt}}}}}"', curve={height=12pt}, hook, from=3-1, to=3-3]
	\arrow[hook, from=3-3, to=1-3]
	\arrow[""{name=3, anchor=center, inner sep=0}, "{{{{\Ocal}}}}"', curve={height=12pt}, from=3-3, to=3-1]
	\arrow["\dashv"{anchor=center, rotate=-90}, draw=none, from=1, to=0]
	\arrow["\simeq"{description}, draw=none, from=3, to=2]
\end{tikzcd}\]
This result can be intended as a form of (strong) conceptual completeness for geometric theories with enough models, provided that we identify the notion of `syntax', for geometric logic, with the topoi classifying geometric theories. 

For subgeometric logics, however, this identification is arguably less natural, as it is subsumes a `geometric completion' of the category naturally built out of the syntax of a theory to its classifying topos. Indeed, mapping a pretopos $\mc C$ to its classifying topos $\mathsf{Sh}(\mc C, J_{\text{coh}})$,
we are \textit{de facto} moving from its coherent theory, made of coherent formulas and sequents, to its full geometric theory, which has many more formulas and sequents.

As a consequence, for instance, the analogous reflection at the level of coherent topoi originally proved by Lurie \cite{lurieUltracategories2018} strays further from Makkai's original meaning of conceptual completeness, which lives entirely within coherent logic by identifying syntactic categories of coherent theories with pretopoi. Our point of view on the above adjunction, in this paper, is to conceive it as a categorification of the classical Isbell duality between locales and topological spaces, from which we can deduce analogous adjunctions for subgeometric logics. 

\begin{rem}[Other approaches to a categorified Isbell duality]
   In this sense, virtual ultracategories represent an \textit{arena of formal model theory} in the sense of \cite{dilibertiFormalModelTheory2024} which, as of today, beautifully brings the \textit{spirit} of point-set topology into the treatment of model-theoretical matters. Moreover, compared to earlier approaches to an Isbell-like duality, such as via \emph{bounded ionads} \cite{dilibertiHigherTopology2022} or \emph{topological profiles} \cite{dilibertiGeometryCoherentTopoi2022}, the environment of virtual ultracategories is more mature for conceptual-completeness-like results. While the formalism of bounded ionads delivers an analogous idempotent adjunction, it does not single out a notion of \emph{closure under ultraproducts} as natively (cf.\ \emph{ultraionad} \cite[Def.\ 3.3.4]{dilibertiGeometryCoherentTopoi2022}). On the other hand, the framework of topological profiles does capture closure under ultraproducts through the notion of \emph{ultraprofiles} \cite[Def.\ 3.2.9]{dilibertiGeometryCoherentTopoi2022}, based on Marmolejo's \emph{\L o\'s categories} \cite{marmolejoUltraproductsContinuousFamilies1995}, but it fails in establishing the idempotency of the analogous adjunction. 
\end{rem}

Focusing on the coherent case, our strategy in this section will be to decompose Makkai's result into a more \emph{syntactic} component, handling definability of geometric concepts in coherent logic, and a more \emph{semantic} component, transforming semantic prescriptions into formulas. Concretely, our starting point will be to realise \Cref{thm:makkai} as a reflective dual adjunction, of the 2-category $\ms{Pretopoi}$ of pretopoi and pretopos functors into an appropriate 2-category of ultracategories and ultrafunctors, singled out among virtual ultracategories and virtual ultrafunctors:
\[\begin{tikzcd}
	{\ms{Pretopoi}\op } && {\mathsf{Ult^{bnd}_{sob}}}
	\arrow[""{name=0, anchor=center, inner sep=0}, "{\ms{Mod}}"', curve={height=12pt}, hook, from=1-1, to=1-3]
	\arrow[""{name=1, anchor=center, inner sep=0}, "{\ms{Ult}(-,\Set)}"', curve={height=12pt}, from=1-3, to=1-1]
	\arrow["\dashv"{anchor=center, rotate=-90}, draw=none, from=1, to=0]
\end{tikzcd}\]
Then, denoting by $\WRInj(\mc H_\beta)_{\ms{wep}}$ the full subcategory of $\WRInj(\mc H_\beta)$ spanned by topoi with enough points, we will show that the above adjunction is a reflection if and only if so is another composite adjunction, depicted below:
\[\begin{tikzcd}
	{\mathsf{Pretopoi}\op} & {\WRInj(\Hcal_{\beta})_{\ms{wep}}} & {\WRInj(\mc M_\beta)^{\ms{bnd}}_{\ms{sob}}}
	\arrow[""{name=0, anchor=center, inner sep=0}, "{{\Cl^{\Hcal_{\beta}}}}"', curve={height=18pt}, from=1-1, to=1-2]
	\arrow[""{name=1, anchor=center, inner sep=0}, "{{{{\Syn^{\Hcal_{\beta}}}}}}"', curve={height=18pt}, from=1-2, to=1-1]
	\arrow[""{name=2, anchor=center, inner sep=0}, "{{{{\mathsf{pt}}}}}"', curve={height=18pt}, from=1-2, to=1-3]
	\arrow[""{name=3, anchor=center, inner sep=0}, "{{{{\mc O}}}}"', curve={height=18pt}, from=1-3, to=1-2]
	\arrow["\dashv"{anchor=center, rotate=-90}, draw=none, from=1, to=0]
	\arrow["\simeq"{description}, draw=none, from=3, to=2]
\end{tikzcd}\]
In this picture, the equivalence on the right will follow from a formal analysis of \emph{semantic prescriptions} in the framework of virtual ultracategories (\cref{semanticprescription}). Thus, the composite adjunction will be a reflection if and only if so is the one on the left, and hence if and only if  $\Hcal_\beta$ is conceptually complete. Summing up, conceptual completeness of $\Hcal_\beta$ in the sense of \Cref{def:cc-logic} will be equivalent to conceptual completeness of coherent logic in Makkai's sense, which is thus dissected into a purely syntactic component and an equivalence induced by completeness with respect to set-based models. 

\begin{rem}[There are no size issues]
Formally, we will adhere to the axiomatisation of ultracategories of \cite{saadiaExtendingConceptualCompleteness2025, aristoteProfunctorialAlgebras2026}. This way, we can speak of a genuine reflection, of {small} pretopoi into bounded (sober) ultracategories, instead of a \emph{reflection in the small} in the sense of \cite[\S 8]{makkaiStoneDualityFirst1987}.
\end{rem}

\subsection{Recalls on virtual ultracategories}\label{virtualultracat}
For the sake of readability, we shall recall here the necessary background on the theory of virtual ultracategories: besides \cite{saadiaExtendingConceptualCompleteness2025, hamadGeneralisedUltracategoriesConceptual2025, vangoolToposesEnoughPoints2026}, we also refer the reader to \cite{aristoteProfunctorialAlgebras2026}, whose conventions we will follow for reasons of brevity. 

\begin{defn}
    For a category $\mc C$, an \emph{ultrafamily} in $\mc C$ is a triple $(I, c, \mu)$ of a set $I$, a functor $c \colon I \to \mc C$, and an ultrafilter $\nu \in \beta I$, also denoted as $(c_i)_{i:\mu}$. Ultrafamilies in $\mc C$ can be assembled into a category $\bbbeta \mc C$ such that the assignment $\mc C \mapsto \bbbeta \mc C$ yields the \emph{ultracompletion} pseudomonad $\bbbeta \colon\CAT \to \CAT$.
\end{defn}

\begin{defn}
    We define \emph{ultracategories}, \emph{ultrafunctors}, and \emph{transformations}, respectively as pseudo-$\bbbeta$-algebras, pseudomorphisms, and algebra 2-cells, and we denote by $\mathsf{Ult}$ their 2-category.
\end{defn}

\begin{defn}
    We define \emph{virtual ultracategories}, \emph{virtual ultrafunctors}, and \emph{transformations}, respectively as profunctorial $\bbbeta$-algebras, representable colax morphisms, and algebra 2-cells, and we denote by $\mathsf{vUlt}$ their 2-category. 
\end{defn}

Explicitly, a virtual ultracategory $\mathscr X$ is defined by a category $X_0$ and a profunctor $\Xi \colon \bbbeta X_0 \pro X_0$, where elements of $\Xi(x, (I,y,\mu))$ are called \emph{ultra-arrows} and denoted by $u \colon x \ult_{\mu} y_i$, endowed with:
\begin{enumerate}
    \item an \emph{identity} $\id_x \colon x \ult_{1} x$ for each $x\in X_0$, where $1$ is the unique ultrafilter on the one-element set;
    \item a \emph{composite} $(v_i)_{\mu}\circ u \colon x \ult_{\nu_i \otimes \mu} z_{i,j}$ for each ultra-arrow $u \colon v \ult_\mu y_i$ and each ultrafamily of ultra-arrows $(v_i \colon y_i \ult_{\nu_i} z_{i,j})_{i:\mu}$, where for $(\nu_i \in \beta J_i)_{i:\mu}$ we denote by $\nu_i \otimes \mu$ the ultrafilter on $\coprod_{i}J_i$ defined by 
    \[ S \in \nu_i \otimes \mu \Leftrightarrow \set{\, i \mid  S \cap J_i \in \nu_i \,} \in \mu .\]
\end{enumerate}
This data is required to satisfy appropriate axioms that we omit here. Virtual ultrafunctors, in these terms, are functors endowed with a suitably-natural action on ultra-arrows which preserves identities and composites.

\begin{rem}[Ultracategories as virtual ultracategories]\label{rem:ultcats-as-vultcats}
   By \cite[Cor.\ 5.2.6]{aristoteProfunctorialAlgebras2026}, we can characterise ultracategories as those virtual ultracategories $\mathscr{X}$ whose defining profunctor $\bbbeta X_0 \pro X_0$ is represented by a pseudo-$\bbbeta$-algebra $\bbbeta X_0 \to X_0$. Note, however, that the inclusion $\mathsf{Ult} \inj \mathsf{vUlt}$ is \emph{not} full: virtual ultrafunctors between ultracategories play the role of the \emph{left ultrafunctors} in \cite{lurieUltracategories2018}.
\end{rem}

\begin{exa}[Topological spaces are virtual ultracategories]
    Every topological space $X$ can be seen as a virtual ultracategory with exactly one ultra-arrow $x \ult_\mu y_i$ if and only if every open neighbourhood $U\subseteq X$ of $x$ satisfies $\set{\, i \mid y_i \in U\,} \in \mu$, and none otherwise. This construction yields a fully faithful embedding $\ms{TopSp} \hook \ms{vUlt^{bnd}}$.
\end{exa}

For every virtual ultracategory $\mathscr X$, the category $\mc O(\mathscr X) \coloneqq \mathsf{vUlt}(\mathscr X, \Set)$ is an infinity-pretopos: we say that $\mathscr X$ is \emph{bounded} if $\mc O(\mathscr X)$ admits a small generator, that is, if it is actually a topos. We denote by $\mathsf{vUlt^{bnd}}\subseteq \mathsf{vUlt}$ and $\mathsf{Ult^{bnd}}\subseteq \mathsf{Ult}$ the full sub-2-categories spanned by bounded virtual ultracategories and bounded ultracategories, respectively. The main result of \cite{saadiaExtendingConceptualCompleteness2025, hamadGeneralisedUltracategoriesConceptual2025, vangoolToposesEnoughPoints2026} can then be stated as follows.

\begin{thm}\label{thm:main-vucats}
    There is an idempotent 2-adjunction
\[\begin{tikzcd}
	\Topoi && {\mathsf{vUlt^{bnd}}}
	\arrow[""{name=0, anchor=center, inner sep=0}, "{\mathsf{pt}}"', curve={height=12pt}, from=1-1, to=1-3]
	\arrow[""{name=1, anchor=center, inner sep=0}, "{\Ocal}"', curve={height=12pt}, from=1-3, to=1-1]
	\arrow["\dashv"{anchor=center, rotate=-90}, draw=none, from=1, to=0]
\end{tikzcd}\]
whose counit, at a topos $\mc E$, is an equivalence if and only if $\mc E$ has enough points.
\end{thm}

We then denote by $\mathsf{vUlt^{bnd}_{sob}}\subseteq \mathsf{vUlt^{bnd}}$ and $\mathsf{Ult^{bnd}_{sob}}\subseteq \mathsf{Ult^{bnd}}$ the full sub-2-categories spanned by \emph{sober} virtual ultracategories and ultracategories, meaning those lying in the essential image of the 2-functor $\ms{pt} \colon \Topoi \to \ms{vUlt}$. Following our strategy described above, note that $\Set$ acts as a dualising object for a 2-adjunction between pretopoi and ultracategories:
\[\begin{tikzcd}
	{\ms{Pretopoi}\op } && {\mathsf{Ult^{bnd}_{sob}}}
	\arrow[""{name=0, anchor=center, inner sep=0}, "{\ms{Mod}}"', curve={height=12pt}, from=1-1, to=1-3]
	\arrow[""{name=1, anchor=center, inner sep=0}, "{\ms{Ult}(-,\Set)}"', curve={height=12pt}, from=1-3, to=1-1]
	\arrow["\dashv"{anchor=center, rotate=-90}, draw=none, from=1, to=0]
\end{tikzcd}\]
This way, \Cref{thm:makkai} is equivalent to the above adjunction being a reflection, i.e.\ $\ms{Mod} \colon \ms{Pretopoi}\op \to \ms{Ult^{bnd}_{sob}}$ being 2-fully-faithful. Note here that, for a small pretopos $\mc A$, we can identify the ultracategory $\ms{Mod}(\mc A)$ with $\ms{pt}(\sh(\mc A, J_{\mr{coh}}))$.

\subsection{Revisiting semantic prescriptions}\label{semanticprescription}
In \cite{dilibertiLogicConcepts2category2025} the authors sometimes refer to their notion of logic as \textit{semantic prescriptions}. Through virtual ultracategories, we can make this intuition even more transparent.

As for topoi, fixed a class $\mathcal M$ of morphisms in $\ms{vUlt^{bnd}}$, we can consider the $2$-category $\WRInj(\mc M)$ of those virtual ultracategories which are right Kan injective with respect to $\mc M$. The theory of right Kan injectivity allows us to specify virtual ultracategories with properties of interest for the \textit{working model theorist}. 

\begin{exa}[Complete virtual ultracategories]\label{ex:compl-vucats}
    For concreteness, one can visualise the special case in which $\mc M$ is the class $\mc M^{\text{all}}$ of morphisms of the type $\mc D \to 1$, where $\mc D$ is a category --- seen as an \emph{Alexandroff} virtual ultracategory as in \cite[Ex.\ 4.2.(vi)]{saadiaExtendingConceptualCompleteness2025} --- and $1$ is the terminal virtual ultracategory. In this case, $\WRInj(\mc M^{\text{all}})$ is precisely the $2$-category of those virtual ultracategories whose categories of points (in the sense of \cite[Def.\ 4.6]{saadiaExtendingConceptualCompleteness2025}) are complete. 
\end{exa}

\begin{exa}[Virtual ultracategories with products] \label{exa:regular}
    Proceeding as in \Cref{ex:compl-vucats}, we can also prescribe the existence of products (resp.\ finite products) by restricting to the class $\mc M^{\text{disc}}$ (resp.\ $\mc M^{\text{fdisc}}$) of morphisms of the type $S \to 1$, where $S$ is a set (resp.\ finite set). For instance,  virtual ultracategories of models of regular theories lie in $\WRInj(\mc M^{\text{disc}})$, see \cite[Lem.\ D2.4.3]{johnstoneSketchesElephantTopos2002} (cf.\ also \cite[Lem.\ 4.3.5]{dilibertiLogicConcepts2category2025}). 
\end{exa}
Through the theory of \textit{logics} in the sense of \cite{dilibertiLogicConcepts2category2025}, we can now transfer these properties into semantic prescriptions at the level of syntax, represented by topoi. Indeed, denoting by $\mc O (\mc M)$ the class of morphisms in $\Topoi$ defined by $\set{ \mc O (f) \, |\, f \in \mc M}$, the adjunction above formally restricts to an equivalence at the level of injectives, as shown below.

\begin{prop}\label{prop:formal-equivalence-injectivity-vult}
  For any class $\mc M$ of maps in $\ms{vUlt^{bnd}_{sob}}$, the 2-adjunction of \Cref{thm:main-vucats} restricts to an equivalence 
\[\begin{tikzcd}
	{\WRInj(\mc O(\mc M))_{\ms{wep}}} && {\WRInj(\mc M)_{\ms{sob}}^{\ms{bnd}}}
	\arrow[""{name=0, anchor=center, inner sep=0}, "{{\mathsf{pt}}}"', curve={height=12pt}, from=1-1, to=1-3]
	\arrow[""{name=1, anchor=center, inner sep=0}, "{{\Ocal}}"', curve={height=12pt}, from=1-3, to=1-1]
	\arrow["\simeq"{description}, draw=none, from=1, to=0]
\end{tikzcd}\]
where $\WRInj(\mc O(\mc M))_{\ms{wep}}$ denotes the full sub-2-category of $\WRInj(\mc O(\mc M))$ spanned by topoi with enough points.
\end{prop}
\begin{proof}
By \Cref{thm:main-vucats}, it is straightforward to see that $\mathsf{pt} \colon \Topoi_{\ms{wep}} \hook \ms{vUlt}$ restricts to a 2-functor $\WRInj(\mc O(\mc M))_{\ms{wep}} \to \WRInj(\mc M)^{\ms{bnd}}_{\ms{sob}}$. The fact that this restriction remains 2-fully-faithful follows also by \Cref{thm:main-vucats}: for a geometric morphism $g \colon \mc E\to \mc F$ in $\Topoi_{\ms{wep}}$, the virtual ultrafunctor $\ms{pt}(g) \colon \ms{pt}(\mc E) \to \ms{pt}(\mc G)$ preserves right Kan extensions along any $f \colon \mathscr{X}\to \mathscr{Y}$ in $\mc M$ if \emph{and only if} $g$ preserves right Kan extensions along $f$.
\[\begin{tikzcd}
	{\mc O(\mathscr{X})} & {\mc E} \\
	{\mc O(\mathscr{Y})} & {\mc F}
	\arrow[from=1-1, to=1-2]
	\arrow["{{\mc O (f)}}"', from=1-1, to=2-1]
	\arrow["g", from=1-2, to=2-2]
	\arrow[""{name=0, anchor=center, inner sep=0}, dashed, from=2-1, to=1-2]
	\arrow[""{name=1, anchor=center, inner sep=0}, dashed, from=2-1, to=2-2]
	\arrow["\cong"{description}, draw=none, from=0, to=1]
\end{tikzcd} \qquad \begin{tikzcd}
	{\mathscr{X}} & {\ms{pt}(\mc E)} \\
	{\mathscr{Y}} & {\ms {pt}(\mc F)}
	\arrow[from=1-1, to=1-2]
	\arrow["f"', from=1-1, to=2-1]
	\arrow["{\ms {pt}(g)}", from=1-2, to=2-2]
	\arrow[""{name=0, anchor=center, inner sep=0}, dashed, from=2-1, to=1-2]
	\arrow[""{name=1, anchor=center, inner sep=0}, dashed, from=2-1, to=2-2]
	\arrow["\cong"{description}, draw=none, from=0, to=1]
\end{tikzcd}\]
By definition of sober virtual ultracategories, $\ms{pt}$ is also essentially surjective, and thus an equivalence.
\end{proof}

\begin{rem}[Semantic prescriptions for regular logic?]
As we have seen in \Cref{exa:regular}, virtual ultracategories of models of regular theories lie in $\WRInj(\mc M^{\text{disc}})$. In fact, since regular functors also induce functors preserving products at the level of models, we have a 2-functor
\[ \ms{Mod} \colon \ms{Exact}\op \to \WRInj(\mc M^{\text{disc}})\]
on the 2-category of (small) exact categories. As products are a crucial feature of regular logic, one could imagine that the 2-category $\WRInj(\mc M^{\text{disc}})^{\ms{bnd}}_{\ms{sob}}$ brings us closer to a conceptual-completeness-like result for regular logic in the spirit of Makkai's one --- that is, based on endowing categories of models with additional structure. However, we will see in \Cref{exa:regular-logic-à-la-makkai} that the correct class of virtual ultrafunctors that delivers such a set-based conceptual completeness result for regular logic is induced by the class of $\alpha$-maps already considered in \Cref{ssec:cc-regular-logic}.
\end{rem}

\subsection{Back to Makkai's conceptual completeness} \label{equivalenttomakkai}
We are finally in the position to recover Makkai's conceptual completeness theorem for coherent logic by proving its equivalence to conceptual completeness of the logic $\Hcal_{\beta}$ in the sense of \Cref{def:cc-logic}. 

\looseness=-1
As in \Cref{ssec:cc-coherent-logic}, for each set $I$, consider the continuous map $i_I \colon I \hook \beta I$ including the discrete space on $I$ into its Stone-Čech compactification. We denote by $\mc M_\beta$ be the class of virtual ultrafunctors defined by the family of these continuous maps.

\begin{cor}\label{cor:beta-injective-vults}
The 2-adjunction of \Cref{thm:main-vucats} restricts to an equivalence:
\[\begin{tikzcd}
	{\WRInj(\Hcal_{\beta})_{\ms{wep}}} && {\WRInj(\mathcal{M}_\beta)^{\ms{bnd}}_{\ms{sob}}}
	\arrow[""{name=0, anchor=center, inner sep=0}, "{{\mathsf{pt}}}"', curve={height=12pt}, from=1-1, to=1-3]
	\arrow[""{name=1, anchor=center, inner sep=0}, "{{\mc O}}"', curve={height=12pt}, from=1-3, to=1-1]
	\arrow["\simeq"{description}, draw=none, from=1, to=0]
\end{tikzcd}\]
\end{cor}
\begin{proof}
Note that $\mc O(\mc M_\beta)$ is precisely the logic $\Hcal_{\beta}$: indeed, on topological spaces and continuous maps, $\mc O \colon \mathsf{vUlt^{bnd}}\to\Topoi$ acts as the functor taking categories of sheaves (see \cite[Rem.\ 4.11]{vangoolToposesEnoughPoints2026}), so that $\mc O(i_I \colon I \hook \beta I )$ is the geometric morphism $i_I \colon \Set^I\hook \sh(\beta I)$. The claim then follows by \Cref{prop:formal-equivalence-injectivity-vult}.
\end{proof}

\looseness=-1
Note then that ultracategories, seen as virtual ultracategories as described in \Cref{rem:ultcats-as-vultcats}, are $\mathcal{M}_\beta$-injective. Crucially, ultrafunctors can be equivalently singled out among virtual ultrafunctors as those preserving right Kan extension along $\mathcal{M}_\beta$-morphisms.

\begin{lem}\label{lem:ultracategories-are-injectives}
The 2-category $\mathsf{Ult}$ embeds 2-fully-faithfully in $\WRInj(\mc M_\beta)$.
\begin{proof}
    Let $\mathscr{X}$ be an ultracategory, defined by a $\bbbeta$-algebra functor $A \colon \bbbeta \mathscr{X} \to \mathscr{X}$, and recall by \Cref{rem:ultcats-as-vultcats} that we can see it as a virtual ultracategory. For a set $I$ and a virtual ultrafunctor $z \colon I \to \mathscr{X}$, i.e.\ simply a functor $z \colon I \to X_0$ into the underlying category of $\mathscr{X}$, the value of the right Kan extension $\ran_{i_I} z$ at an ultrafilter $\mu \in \beta I$ is given by the ultraproduct of the ultrafamily $(z_i)_{i:\mu}$ in $\mathscr{X}$ (cf.\ \cite[Cons.\ 2.0.1]{dilibertiGeometryCoherentTopoi2022}), meaning formally:
    \[ \ran_{i_I} z (\mu) \coloneqq A ( I, z\colon I \to X_0, \mu\in\beta I). \]
    It is then straightforward to see that a virtual ultrafunctor between ultracategories preserves these right Kan extensions if and only if it is an ultrafunctor.
\end{proof}
\end{lem}

Recall now that, by the results of \Cref{ssec:cc-coherent-logic}, classifying topoi for $\Hcal_\beta$ are computed as the usual topoi of sheaves over pretopoi: in particular, they have enough points by Deligne's theorem \cite[Thm.\ D3.3.13]{johnstoneSketchesElephantTopos2002}. Therefore, the 2-adjunction $\Syn^{\Hcal_{\beta}}\dashv \Cl^{\Hcal_{\beta}}$ corestricts to $\Hcal_{\beta}$-injectives with enough points, which means that we can draw the desired diagram: 
\[\begin{tikzcd}
	{\mathsf{Pretopoi}\op} & {\WRInj(\Hcal_{\beta})_{\ms{wep}}} & {\WRInj(\mc M_\beta)\mathsf{^{bnd}_{sob}}}
	\arrow[""{name=0, anchor=center, inner sep=0}, "{{\Cl^{\Hcal_{\beta}}}}"', curve={height=18pt}, from=1-1, to=1-2]
	\arrow[""{name=1, anchor=center, inner sep=0}, "{{{{\Syn^{\Hcal_{\beta}}}}}}"', curve={height=18pt}, from=1-2, to=1-1]
	\arrow[""{name=2, anchor=center, inner sep=0}, "{{{{\mathsf{pt}}}}}"', curve={height=18pt}, from=1-2, to=1-3]
	\arrow[""{name=3, anchor=center, inner sep=0}, "{{{{\mc O}}}}"', curve={height=18pt}, from=1-3, to=1-2]
	\arrow["\dashv"{anchor=center, rotate=-90}, draw=none, from=1, to=0]
	\arrow["\simeq"{description}, draw=none, from=3, to=2]
\end{tikzcd}\]
Note that the 2-functor $\ms{Mod}\colon \ms{Pretopoi}\op \to \ms{Ult^{bnd}_{sob}}$ considered above is then a corestriction of the composite $\mathsf{pt} \circ \Cl^{\Hcal_\beta}$ to ultracategories and ultrafunctors, since we showed in \Cref{ssec:cc-coherent-logic} that $\Cl^{\Hcal_{\beta}}$ computes the usual topoi of sheaves over pretopoi. Therefore, since the inclusion $\ms{Ult^{bnd}_{sob}}\hook \WRInj(\mc M_\beta)^{\ms{bnd}}_{\ms{sob}}$ is full by \Cref{lem:ultracategories-are-injectives}, the following are equivalent:
\begin{itemize}
    \item[i.]$\ms{Mod}\colon \ms{Pretopoi}\op \to \ms{Ult^{bnd}_{sob}}$ is 2-fully-faithful;
    \item[ii.]$\ms{Mod}\colon \ms{Pretopoi}\op \to \WRInj(\mc M_{\beta})_{\ms{sob}}^{\ms{bnd}}$ is 2-fully-faithful;
    \item[iii.]$\Cl^{\Hcal_\beta} \colon \ms{Pretopoi}\op \to \WRInj(\Hcal_\beta)_{\ms{wep}}$ is 2-fully-faithful;
    \item[iv.]$\Cl^{\Hcal_\beta} \colon \ms{Pretopoi}\op \to \WRInj(\Hcal_\beta)$ is 2-fully-faithful.
\end{itemize}
Looking at (i) and (iv), we finally conclude that conceptual completeness of coherent logic in the sense of \cite{makkaiStoneDualityFirst1987} is equivalent to conceptual completeness of the logic $\Hcal_{\beta}$ in the sense of \Cref{def:cc-logic} --- which we proved in \Cref{ssec:cc-coherent-logic}.

\begin{cor}\label{cor:recovering-makkai}
    The following are equivalent:
    \begin{enumerate}
        \item the 2-adjunction $\ms{Mod} \colon \begin{tikzcd}[cramped]
        {\ms{Pretopoi}\op} & {\ms{Ult^{bnd}_{sob}}}
        \arrow[""{name=0, anchor=center, inner sep=0}, shift right=1.3, from=1-1, to=1-2]
        \arrow[""{name=1, anchor=center, inner sep=0}, shift right=1.7, from=1-2, to=1-1]
        \arrow["\dashv"{anchor=center, rotate=-90}, draw=none, from=1, to=0]
    \end{tikzcd} \cocolon {\ms{Ult}(-,\Set)}$ is a reflection, i.e.\ \Cref{thm:makkai} holds;
    \item the 2-adjunction $\Cl^{\Hcal_\beta} \colon \begin{tikzcd}[cramped]
        {\ms{Pretopoi}\op} & {\WRInj(\Hcal_\beta)}
        \arrow[""{name=0, anchor=center, inner sep=0}, shift right=1.3, from=1-1, to=1-2]
        \arrow[""{name=1, anchor=center, inner sep=0}, shift right=1.7, from=1-2, to=1-1]
        \arrow["\dashv"{anchor=center, rotate=-90}, draw=none, from=1, to=0]
    \end{tikzcd} \cocolon {\Syn^{\Hcal}}$ is a reflection, i.e.\ $\Hcal_\beta$ is conceptually complete.
    \end{enumerate}
\end{cor}

\begin{rem}
Compare this discussion with \cite[Rem.\ 6.3.5, 6.3.6]{dilibertiLogicConcepts2category2025}: we have now justified our notion of conceptual completeness by showing its equivalence, in the coherent case, to Makkai's notion, and we managed to deliver a proof of Makkai's result which is contained in our framework.
\end{rem}

\subsection{A modular account of conceptual completeness \emph{à la} Makkai}\label{ssec:cc-à-la-makkai}{\looseness=-1 The discussion in the previous subsection for the logic $\Hcal_\beta$ allows us to define a notion of conceptual completeness, for a (bounded) logic $\Hcal$, grounded in its $\Set$-based semantics. Indeed, suppose that $\Hcal = \mc O(\mc M)$ for some class $\mc M$ of morphisms in $\ms{vUlt^{bnd}_{sob}}$, so that we can think of $\WRInj(\mc M)$ as the `correct' semantic prescriptions to study conceptual completeness of $\Hcal$.
The composite 2-functor 
\[\begin{tikzcd}
	{\alg(\T^{\Hcal})\op} & {\WRInj(\Hcal)} & {\WRInj(\mc M)}
	\arrow["{\Cl^{\Hcal}}", from=1-1, to=1-2]
	\arrow["{\mathsf{pt}}", from=1-2, to=1-3]
\end{tikzcd}\]
can always be identified with an appropriate lift to virtual ultracategories of the representable 2-functor $\alg(\T^{\Hcal})\op \to \CAT$ defined by homming into $\Set$, since by Diaconescu's equivalence (\Cref{thm:diaconescu}) we have $\Topoi(\Set, \Cl^{\Hcal}(\mc A)) \simeq \Alg(\T^{\Hcal})(\mc A, \Set)$ for any small $\T^{\Hcal}$-algebra $\mc A$. Thus, as in the coherent case, we denote this 2-functor as
\[ \mathsf{Mod} \colon \alg(\T^{\Hcal})\op \to \WRInj(\mc M).\]

\begin{defn}[Conceptual completeness \emph{à la} Makkai]\label{def:cc-à-la-makkai}
	A bounded logic $\Hcal$ such that $\Hcal = \mc O(\mc M)$ for some class $\mc M$ of morphisms in $\ms{vUlt^{bnd}}$ is \emph{conceptually complete {à la Makkai}} if the 2-functor $\mathsf{Mod} \colon \alg(\T^{\Hcal})\op \to \WRInj(\mc M)$ is fully faithful.
\end{defn}

\looseness=-1
Suppose now that $\Hcal$-classifying topoi have enough points. Then, as in the coherent case, we get the following picture, which we can see as a \emph{syntax-semantics duality} or a \emph{reconstruction theorem} for the logic $\Hcal$, arising via $\Set$ as a dualising object:
\[\begin{tikzcd}
	{\alg(\T^{\Hcal})\op} & {\WRInj(\mc O(\mc M))_{\ms{wep}}} & {\WRInj(\mc M)^{\ms{bnd}}_{\ms{sob}}}
	\arrow[""{name=0, anchor=center, inner sep=0}, "{\Cl^{\Hcal}}"{description}, curve={height=18pt}, from=1-1, to=1-2]
	\arrow["{\alg(\T^{\Hcal})(-,\Set)}"', curve={height=36pt}, from=1-1, to=1-3]
	\arrow[""{name=1, anchor=center, inner sep=0}, "{\Syn^{\Hcal}}"{description}, curve={height=18pt}, from=1-2, to=1-1]
	\arrow[""{name=2, anchor=center, inner sep=0}, "{{{{{\mathsf{pt}}}}}}"{description}, curve={height=18pt}, from=1-2, to=1-3]
	\arrow["{\WRInj(\mc M)(-,\Set)}"', curve={height=36pt}, from=1-3, to=1-1]
	\arrow[""{name=3, anchor=center, inner sep=0}, "{{{{{\mc O}}}}}"{description}, curve={height=18pt}, from=1-3, to=1-2]
	\arrow["\dashv"{anchor=center, rotate=-90}, draw=none, from=1, to=0]
	\arrow["\simeq"{description}, draw=none, from=3, to=2]
\end{tikzcd}\]

\begin{thm}[Recovering a set-based duality]\label{bigthoeremwithtrivialproof}
	Let $\Hcal$ be a bounded logic such that $\Hcal = \mc O(\mc M)$ for some class $\mc M$ of morphisms in $\mathsf{vUlt^{bnd}_{sob}}$, and suppose that $\Hcal$-classifying topoi have enough points.
Then, the following are equivalent:
\begin{enumerate}
	\item $\Hcal$ is conceptually complete.
	\item $\Hcal$ is conceptually complete \emph{à la} Makkai.
\end{enumerate}
\begin{proof}
    Immediate from the above picture, since $\ms{Mod} (-) \cong \Alg (\T^{\Hcal})(-,\Set)$.
\end{proof}
\end{thm}

\begin{exa}[Regular logic \emph{à la} Makkai]\label{exa:regular-logic-à-la-makkai}
	Exactly as for $\Hcal_\beta$, the logic $\Hcal_{\alpha}$ introduced in \Cref{ssec:cc-regular-logic} is conceptually complete \emph{à la} Makkai, which can be intended as a conceptual completeness result in Makkai's sense for \emph{regular} logic.
\end{exa}

\begin{rem}[A final resolution]\label{rem:resolution}
  There are essentially two ways to look at \Cref{bigthoeremwithtrivialproof}. A somewhat negative point of view is to say that our notion of conceptual completeness is weaker than Makkai's and that the two are equivalent only under the additional assumption of a \textit{completeness} theorem. This is an apparently natural and neutral reading of the theorem above; yet, we think that it would be a misunderstanding of our work. What we provided, following the ideas also presented in~\cite{dilibertiLogicConcepts2category2025}, is a finer analysis of conceptual completeness, dissecting it into: 
  \begin{itemize}
      \item[(a)] a \emph{semantic} component which hinges on the conceptual completeness theorem of geometric logic in the sense of \cite{saadiaExtendingConceptualCompleteness2025, hamadGeneralisedUltracategoriesConceptual2025, vangoolToposesEnoughPoints2026}, reifying semantic prescriptions into geometric formulas;
        \item[(b)] a \emph{syntactic} component in the spirit of a definability theorem, reducing geometric formulas into formulas in the fragment.
        \end{itemize}
 \end{rem}

\bibliography{bibliography}
\bibliographystyle{alpha}

\end{document}